\documentclass[12pt,reqno]{amsart}
\usepackage{amsmath}
\usepackage{amssymb}
\usepackage{amstext}
\usepackage{mathrsfs}
\usepackage{a4wide}
\usepackage{graphicx}
\usepackage{bbm}
\usepackage{verbatim}
\usepackage{bm}
\usepackage{graphicx}
\usepackage{tikz}
\usepackage{pgfplots}
\allowdisplaybreaks[1]

\numberwithin{equation}{section}
\usepackage{color}
\usepackage{cases}
\usepackage{hyperref}
\hypersetup{hypertex=true,
	colorlinks=true,
	linkcolor=blue,
	anchorcolor=blue,
	citecolor=blue }

\numberwithin{equation}{section}

\newtheorem{theorem}{Theorem}[section]
\newtheorem{proposition}[theorem]{Proposition}

\newtheorem{lemma}[theorem]{Lemma}

\theoremstyle{definition}

\newtheorem{definition}[theorem]{Definition}
\newtheorem{mytheorem}{Theorem}[section]

\theoremstyle{remark}
\newtheorem{remark}[theorem]{Remark}

\newcommand{\Om}{\Omega}

\newcommand{\R}{\mathbb{R}}

\usepackage{titletoc}
\usepackage[T1]{fontenc}
\usepackage{dsfont}

\newcommand{\Ran}{\operatorname{Range}}

\newcommand{\lset}[1]{\left\{ #1 \right.}

\newcommand{\Span}{\operatorname{span}}

\newcommand{\N}{\mathbb{N}}
\newcommand{\T}{\mathbb{T}}
\newcommand{\Z}{\mathbb{Z}}
\usepackage{extarrows}

\usepackage{autobreak}

\begin{document}

	\title{SYMMETRY OF UNIFORMLY ROTATING SOLUTIONS FOR THE VORTEX-WAVE SYSTEM }

		\author{Daomin Cao, Boquan Fan, Rui Li}

		\address{Institute of Applied Mathematics, Chinese Academy of Sciences, Beijing 100190, and University of Chinese Academy of Sciences, Beijing 100049,  P.R. China}
		\email{dmcao@amt.ac.cn}
		\address{Institute of Applied Mathematics, Chinese Academy of Sciences, Beijing 100190, and University of Chinese Academy of Sciences, Beijing 100049,  P.R. China}
		\email{fanboquan22@mails.ucas.ac.cn}
		\address{Institute of Applied Mathematics, Chinese Academy of Sciences, Beijing 100190, and University of Chinese Academy of Sciences, Beijing 100049,  P.R. China}
		\email{lirui2021@amss.ac.cn}

    \begin{abstract}
		In this paper, we study the radial symmetry properties of stationary and uniformly rotating solutions of the vortex-wave system introduced by Marchioro and Pulvirenti \cite{Mar1}. We show that every uniformly rotating patch $\left(D,x_1,x_2,..,x_k\right)$ with angular velocity $\Omega\leq 0$ must be radial with respect to the only point vortex $x_1$, implying that $k=1$. In other words, the background vorticity consists of finite nested annulus and the point vortex is located at the center of these annulus. 
        In contrast to the case where the angular velocity is non-positive, we prove that there exists a family of uniformly rotating patch $(D^n,x_1^n)_n$ solutions, which are associated with a sequence of positive angular velocities $\{\Om_n\}$ and are not annular. Furthermore, we find that the set of bifurcating angular velocities $\{\Om_n\}$ is dense in the interval $(0,+\infty)$, a novel feature that distinguishes this behavior from that observed in the classical Euler equation and gSQG equation.
      
    \end{abstract}

	\maketitle


\section{Introduction and main results}
	\subsection{Vortex-wave system}
        In this paper, we focus on the vortex-wave system, representing the complex interactions between the two-dimensional Euler vorticity equation and the point-vortex system. To this end, we begin by delving into the fundamentals of the 2D Euler equation, which describes the dynamics of an incompressible, inviscid fluid in $\R^2$ and can be written as:
        \begin{align}\label{1-1}
            \begin{cases}
                \partial_t\textbf{u}+\left(\textbf{u}\cdot \nabla\right) \textbf{u} =-\nabla P,&\text{in}\ \mathbb{R}^2\times (0,T),\\
                \ \nabla\cdot\textbf{u}=0     ,&\text{in}\ \mathbb{R}^2\times (0,T),
            \end{cases}
        \end{align}
        where $\textbf{u}=\left(u_1,u_2\right)$ is the velocity field and $P$ is the scalar pressure. 
        In the planar case, we usually introduce  the scalar vorticity of the fluid as 
        \begin{equation*}
            \omega:=\partial_{x_1}u_2-\partial_{x_2}u_1.
        \end{equation*}

        By applying the curl operator to both sides of the first equation in \eqref{1-1}, we derive the vorticity formulation of the Euler equations as follows (see \cite{Maj,Mar2}):
        \begin{align}\label{1-2}
            \partial_t\omega+\textbf{u}\cdot\nabla\omega=0,\quad \text{in}\ \R^2\times (0,T).
        \end{align}

        The velocity field can be reconstructed from the vorticity by using the well-known Biot-Savart law, expressed as:
        \begin{equation}\label{1-3}
            \textbf{u}(x,t)=-\frac{1}{2\pi}\int_{\mathbb{R}^2}\frac{(x-y)^\perp}{|x-y|^2}\omega(y,t)\,dy,
        \end{equation}
        where $(a,b)^\perp=(b,-a)$ denotes clockwise rotation through $\pi/2$ of any vector $(a,b)\in\mathbb{R}^2$. We will denote $K(x)=-(1/2\pi)(x^\perp/|x|^2)$ as the Biot-Savart kernel in the sequel. 
        Equations \eqref{1-2} and \eqref{1-3} indicate that the vorticity $\omega$ is advected by a velocity field that remains divergence-free, induced by its own vorticity.

        In certain scenarios where the vorticity is sharply concentrated in $ k$ small disjoint regions, with $k$ being a positive integer, its temporal evolution can be approximated by the point vortex model (as exemplified in \cite{Lin}). This model is represented by an ordinary differential equation (ODE) system, which can be expressed as:
        \begin{equation*}
            \frac{dx_i}{dt}=\sum_{j=1,j\neq i}^{k}\kappa_jK(x_i-x_j),\quad {\rm for}\ i=1,...,k.
        \end{equation*}
        The point vortex model and its relation with the Euler equations have been analyzed extensively; see, for example, \cite{Cao2015,Mar 1986,Schaftingen 2010,Turkington 1987}.
        Within the framework of the point vortex model, for each point vortex located at $x_i$, with a corresponding vorticity strength $\kappa_i$, the induced velocity is given by the function $\kappa_i K(~\cdot - x_i)$. Additionally, the dynamic evolution of each point vortex is determined by the superposition of velocity fields generated by all other $k - 1$ point vortices present in the model.

        Based on the aforementioned considerations, Marchioro and Pulvirenti introduced and investigated a hybrid model in \cite{Mar1}, commonly referred to as the \textit{Vortex-Wave System}. In the whole plane the vortex-wave system can be written as:
        \begin{equation}\label{1-4}
            \begin{cases}
                \partial_t\omega+\textbf{u}\cdot\nabla\omega=0,\vspace{0.25em}\\
                \ \frac{dx_i}{dt}=K*\omega(x_i,t)+\sum_{j=1,j\neq i}^{k}\kappa_jK(x_i-x_j),\quad {\rm for}\ i=1,...,k,\vspace{0.25em}\\
                \ \textbf{u}=K*\omega+\sum_{j=1}^{k}\kappa_jK(~\cdot-x_j).
            \end{cases}
        \end{equation}

        Throughout this paper, we call $\omega:\R^2\to \R$ the background vorticity and always assume $\kappa_j=1$, for $j=1,...,k$. Then system \eqref{1-4} indicates that the background vorticity is transported both by the velocity field induced by itself (the term $K * \omega$) and by the $k$ point vortices (the term $\sum_{j=1}^{k} \kappa_j K(~\cdot - x_j)$). Furthermore, each point vortex moves according to the velocity induced by the background vorticity (the term $K * \omega(x_i, t)$), as well as by the combined effect of all other $k - 1$ point vortices (the term $\sum_{\substack{j=1, j \neq i}}^{k} \kappa_j K(~\cdot - x_j)$). 
        In \cite{CaoWang2019}, Cao and Wang show the existence of steady multiple vortex patch solutions to the vortex-wave system in a planar bounded domain. 
        More results on the existence and uniqueness can be found in \cite{Bjorland,LM,LMN,Miot}.

	\subsection{Uniformly rotating patch solutions}
        In this paper, our first objective is to establish the radial symmetry properties for stationary and uniformly rotating solutions. 
        Our second objective is to construct non-radial solutions to the vortex-wave system. In particular, we only consider the patch solutions, where $\omega(\cdot,t)=1_{D(t)}$ is an indicator function of a bounded set $D(t)$ which evolves with the fluid.

        Let us begin with the definition of a stationary/uniformly rotating solutions in the patch setting. 
        \begin{definition}
            We call a solution $\left(\omega(x,t),x_1(t),...,x_k(t)\right)$ of \eqref{1-4} is uniformly rotating with angular $\Omega$ if $\omega(x,t)=\omega(R_{-\Omega t}x,0)$ and $x_i(t)=R_{\Omega t}x_i(0)$ for $i=1,...,k$, where $R_{\Omega t}x$ rotates a vector $x\in\mathbb{R}^2$ counterclockwise by angle $\Omega t$ about the origin.
        \end{definition} 

        Assume $\left(\omega(R_{-\Omega t}x),R_{\Omega t}x_1,...,R_{\Omega t}x_k\right)$ is uniformly rotating solution of \eqref{1-4}, then a simple computation gives the following system:
        \begin{align}\label{1-5}
            \begin{cases}
                \nabla^\perp\left(\omega*N+\sum_{i=1}^{k}N(x-x_i)+\frac{\Omega}{2}|x|^2\right)\cdot\nabla\omega=0,\\
                \ \Omega x_i^\perp=\frac{1}{2\pi}\int_{\mathbb{R}^2}\frac{(x_i-y)^\perp}{|x_i-y|^2}\omega(y)dy-\sum_{j=1,j\neq i}^{k}K(x_i-x_j),~{\rm for}\ i=1,...,k,
            \end{cases}
        \end{align}
        where $N(x)=-1/2\pi\ln|x|$. Since we are focusing on the patch solutions setting, we need to introduce a concept of weak solutions in order that \eqref{1-5} makes sense. To do this, multiplying both side of the first equation in \eqref{1-5} by a test function $\phi\in C_0^\infty(\R^2)$, we have
        \begin{equation*}
            \begin{split}
                0=&\int_{\mathbb{R}^2}\phi\nabla^\perp\left(\omega*N+\sum_{i=1}^{k}N(x-x_i)+\frac{\Omega}{2}|x|^2\right)\cdot\nabla\omega ~dx\\
                =&-\int_{\mathbb{R}^2}\phi\nabla\cdot\left[\omega\nabla^\perp\left(\omega*N+\sum_{i=1}^{k}N(x-x_i)+\frac{\Omega}{2}|x|^2\right)\right]dx\\
                =&-\int_{\mathbb{R}^2}\omega\nabla^\perp\left(\omega*N+\sum_{i=1}^{k}N(x-x_i)+\frac{\Omega}{2}|x|^2\right)\cdot\nabla\phi~dx.
            \end{split}
        \end{equation*}
        Based on the above considerations, we give the following definition.

        \begin{definition}
            We say $\left(\omega,x_1,...,x_k\right)\in L^\infty_c(\mathbb{R}^2)\times\left(\mathbb{R}^2\right)^k$ (where $L^\infty_c(\mathbb{R}^2)$ denotes the set of all bounded measurable functions with compact support) is a weak solution of \eqref{1-5} if for any $\phi\in C^\infty_0(\mathbb{R}^2)$, we have
            \begin{align}\label{1-6}
                \begin{cases}
                    \int_{\mathbb{R}^2}\omega\nabla^\perp\left(\omega*N+\sum_{i=1}^{k}N(x-x_i)+\frac{\Omega}{2}|x|^2\right)\cdot\nabla\phi~dx=0,\\
                    \ \Omega x_i^\perp=\frac{1}{2\pi}\int_{\mathbb{R}^2}\frac{(x_i-y)^\perp}{|x_i-y|^2}\omega(y)dy-\sum_{j=1,j\neq i}^{k}K(x_i-x_j),~~i=1,...,k.
                \end{cases}
            \end{align}
        \end{definition}
        Having made these preparations, we can define the uniformly rotating patch solutions of \eqref{1-4}.
        \begin{definition}
            Assume that $D$ is a bounded open set with $C^{1,\alpha}$ boundary and $x_i\in\mathbb{R}^2$ for $i=1,...,k$. We say $\left(D,x_1,...,x_k\right)$ is a uniformly rotating patch solution with angular velocity $\Omega$ of the vortex-wave system \eqref{1-4}, if $\left(1_D,x_1,...,x_k\right)$ is a weak solution of \eqref{1-5}. Furthermore, if $\Omega=0$, then we call it stationary patch solution.
        \end{definition}

	\subsection{Main results} 
        Firstly, let us review some symmetry properties of the planar Euler equation. In this case, if $D\subset\mathbb{R}^2$ is a uniformly rotating solution of the 2D Euler equation with angular velocity $\Omega$, then the function $\psi=1_{D}*N+(\Omega/2)|x|^2$ must be a constant along each connected component of $\partial D$. If $\Omega=0$ and $\psi=constant$ on the whole $\partial D$, then Fraenkel \cite{Fra} proved that  $D$ must be a disk. In that case, Fraenkel showed that the stream function $\phi$ must solve a semilinear elliptic equation $-\Delta\psi=f(\psi)$ in $\mathbb{R}^2$ with $f(\psi)=1_{\left\{\psi>C\right\}}$ and then used the moving-plane method developed in \cite{Gid,Ser} to obtain the symmetry. For $\Omega<0$, Hmidi \cite{Hmi} used the moving-plane as well to show that if a simply-connected uniformly rotating patch $D$ satisfies some additional convexity assumption, then $D$ must be a disk. 
        
        On the other hand, it is known that nonradial uniformly rotating patches may exist for $\Om\in(0,\frac{1}{2})$. A pioneering instance is the Kirchhoff ellipse in \cite{Kirchhoff ellipse}, which demonstrated that any ellipse domain $D$ with semiaxes $a,b$ is a uniformly rotating patch with angular velocity $\frac{ab}{(a+b)^2}$. Follow-up studies, including numerical investigations in \cite{numerics1,numerics2,numerics3} have explored this phenomenon. Burbea \cite{Burbea} provided the first rigorous proof of their existence using (local) bifurcation theory. There have been many recent developments in a series of work by Hmidi, Mateu, Verdera, and de la Hoz (see \cite{Del, Hmidi-boundary,HMV}). Notably, \cite{Del} proved the existence of $m$-fold doubly connected nonradial patches with annular velocities, which are dense in $(0,\frac{1}{2})$. In a more contemporary development, Hmidi, Houamed, Roulley and Zerguine \cite{Hmidi-2024-Lake} found the existence of $m$-fold nonradial uniformly rotating patch solutions to the lake equation, under some assumptions of the depth function.
       
        Following the aforementioned researches, many remaining issues still await resolution. Notably, G$\acute{o} $mez-Serrano, Park, J. Shi and Y. Yao \cite{Serr} employed several novel and powerful methods to settle the radial symmetry properties of stationary and uniformly rotating solutions not only for 2D Euler but also for gSQG equations, both in the smooth solution setting and the patch solution settings. Concerning the 2D Euler equation, they establish that any smooth stationary solution characterized by compactly supported and nonnegative vorticity necessarily exhibits radial symmetry, without imposing constraints on the connectedness of the support or the level sets. Extending their analysis to the patch solution setting of the 2D Euler equation, they demonstrate that every uniformly rotating patch, denoted as $D$, with angular velocities $\Omega$ satisfying $\Omega \leq 0$ or $\Omega \geq \frac{1}{2}$ must possess radial symmetry, with both bounds being proven to be optimal. We restate their result as follows:
        \begin{mytheorem}(=Theorem A in \cite{Serr})\label{Thm A}
            Let $D\subset\mathbb{R}^2$ be a bounded open set with $C^{1,\alpha}$ boundary. Assume that $D$ is a stationary/uniformly rotating patch of \textbf{2D Euler equation} for angular velocity $\Omega$. Then $D$ must be radially symmetric if $\Omega< 0$ or $\Omega\geq 1/2$, and radially symmetric up to a translation if $\Omega=0$.
        \end{mytheorem}
        As for the gSQG equation (the related kernel is $-C_\alpha|x|^{-\alpha}$ for $\alpha\in(0,2)$), they obtain analogous symmetry results for angular velocities $\Omega$ satisfying $\Omega \leq 0$ or $\Omega \geq \Omega^\alpha$, where $\Omega^\alpha$ denotes a certain parameter, with the bounds being shown to be sharp. Notably, this analysis requires the additional assumption that the patch is simply connected. Their approach draws heavily from a calculus-of-variations point of view, facilitating novel insights into these symmetry. We also restate the aforementioned result as follows for the convenience of readers:
        \begin{mytheorem}(=Theorem C, Theorem D in \cite{Serr})\label{Thm B}
            Let $D\subset\mathbb{R}^2$ be a bounded simply connected domain with rectifiable boundary. Assume $D$ is a stationary/uniformly rotating patch of \textbf{gSQG equation} for angular velocity $\Omega$. Then $D$ must be a disk, if $\Om\leq 0$ and $\alpha\in(0,2)$; and $D$ must be the unit disk, if $\Omega\geq\Omega^\alpha$, $\alpha\in(0,1)$ and the area of $D$ is $\pi$.
        \end{mytheorem}
        \begin{remark}
            Hassainia and Hmidi \cite{Has} showed the existence of nonradial patch in the case $\alpha\in(0,1)$. In that case, the solutions bifurcate at angular velocities given by 
            $\Omega^\alpha_m$ which are increasing function of $m$ and its limit is a finite number $\Omega^\alpha$, so the bound $\Omega^\alpha$ is also optimal.
        \end{remark}

        Upon comparing the results of Theorem \ref{Thm A} and Theorem \ref{Thm B}, two fundamental questions naturally arise in the case of \textbf{vortex-wave system}.

    \bigskip
        \textbf{QUESTION 1}
        
        Is the assumption $\Omega\leq 0$ suffice to imply the symmetry of the solutions?
        
     \bigskip
        
        \textbf{QUESTION 2}
    
        For $\Omega>0$, does a critical number ($1/2$ in the Euler case, $\Omega^\alpha$ in the gSQG case) exist?
        
    \bigskip
        Our main results will provide a comprehensive answer to both Question 1 and Question 2. The first two results will affirmatively address Question 1, we state them as follows:

        \begin{theorem}\label{Th1-4}
            Let $D$ be a bounded open set with $C^{1,\alpha}$ boundary  and $x_1,...,x_k$ are $k(\geq 1)$ different points such that $x_i\in \mathbb{R}^2\setminus\overline{D}$ for $i=1,...,k$. Assume that $\left(D,x_1,...,x_k\right)$ is a uniformly rotating patch solution of the vortex-wave system \eqref{1-4} with angular velocity $\Omega$. If $\Omega<0$, then $D$ is radially symmetry with respect to the origin, $k=1$ and $x_1$ is located at the origin. More precisely, $D$ consists of finite concentric and nested annulus centered at $x_1$ (origin).
        \end{theorem}
        \begin{theorem}\label{Th1-5}
            Under the assumption of Theorem \ref{Th1-4}, if $\Omega=0$, then $D$ consists of finite concentric and nested annulus centered at the only point vortex $x_1$ (which may not be the origin).
        \end{theorem}
        \begin{remark}
            In the framework of vortex-wave systems, the implications of solutions' symmetry encompass the symmetry of the background vortex patch region, the number of point vortices, and the relative positions between point vortices and the background vortices.
        \end{remark}
        \begin{remark}
            If $D$ consists of finite concentric and nested annulus centered at the only point vortex $x_1$, then one can check easily that $(D,x_1)$ is a stationary patch solution to the vortex-wave system \eqref{1-4}. Moreover, if $x_1=0$, then $(D,x_1)$ is a uniformly rotating patch solution with any angular velocity $\Omega\in\mathbb{R}$.
        \end{remark}

        Regarding Question 2, by the Crandall-Rabinowitz theorem, we have a negative answer, which also highlights the difference between vortex-wave system and either 2D Euler equation or gSQG equation. To achieve this goal, we need to construct some nonradial solutions to the vortex-wave system \eqref{1-4} with a sequence of angular velocities going to infinity. In fact, we have a stronger result. The set of angular velocities corresponding to the nonradial uniformly rotating solutions is dense in $(0,\infty)$. We state it as follows:
        \begin{theorem}\label{main theorem 2}
            Let $0<a_1<a_2$. Then there exists $N(a_1,a_2)\in \N$ such that, for all $n\geq N(a_1,a_2)$, there are two curves of uniformly rotating patch solutions of the vortex-wave system \eqref{1-4} in the case of $k=1$, bifurcation from the stationary solution 
            \begin{equation*}
                \left(\omega,(x_1^1,x_1^2)\right)=\left(1_{\{x\in R^2:a_1<|x|<a_2\}},(0,0)\right),
            \end{equation*}
            at the angular velocity
            \begin{equation*}
                \Om_{n}^{\pm}=\left(\frac{a_2^2-a_1^2}{4a_2^2}+\frac{1}{4\pi a_2^2}+\frac{1}{4\pi a_1^2}\right)\pm\frac{1}{2}\sqrt{\left(\frac{a_2^2-a_1^2}{2a_2^2}+\frac{1}{2\pi a_2^2}-\frac{1}{2\pi a_1^2}-\frac{1}{n}\right)^2-\frac{1}{n^2}\left(\frac{a_1}{a_2}\right)^{2n}}.
            \end{equation*}
            Moreover,
            \begin{equation}\label{alimit of Om plus}
            	\lim_{n\to \infty}\Om_{n}^+=\max\left\{\frac{a_2^2-a_1^2}{2a_2^2}+\frac{1}{2\pi a_2^2},\ \frac{1}{2\pi a_1^2}\right\}
            \end{equation}
            and
            \begin{equation}\label{alimit of Om minus}
            	\lim_{n\to\infty}\Om_{n}^-=\min\left\{\frac{a_2^2-a_1^2}{2a_2^2}+\frac{1}{2\pi a_2^2},\ \frac{1}{2\pi a_1^2}\right\}.
            \end{equation}
        \end{theorem}
        \begin{remark}
            We can easily show that 
            $\{\Om_n^\pm,\ n\geq N(a_1,a_2),0<a_1<a_2<\infty\}$ is dense in $(0,\infty)$, hence the critical number in Question 2 does not exist. 
            In fact, for $c\in(0,\infty)$, if we take $a_1=\sqrt{1/2\pi c}$, then either $\Omega_n^+\to c$ or $\Omega_n^-\to c$, as $n\to\infty$, by \eqref{alimit of Om plus} and \eqref{alimit of Om minus}.
        \end{remark}

\section{Symmetry of uniformly rotating patch solutions for \texorpdfstring{$\Omega\leq0$}{}}\label{sec-2}
    \subsection{Notations and preliminaries}\label{subsec-2-1}
        In this section, we will give some notations and some lemmas used to prove Theorem \ref{1-4} and Theorem \ref{1-5}.

        Jordan-Schoenflies theorem guarantees that any simple closed curve $\Gamma$ divides $\mathbb{R}^2$ into two domain, one of them is bounded, which is denoted by int$\left(\Gamma\right)$.

        Two disjoint simply closed curves $\Gamma_1$ and $\Gamma_2$ are called nested if $\Gamma_1\subset$ int$\left(\Gamma_2\right)$ or $\Gamma_2\subset$ int$\left(\Gamma_1\right)$. Similarly, two domains $D_1$ and $D_2$ are nested if one is contained in a hole of the other one.

        $1_D$ always denotes the indicator function of $D\subset\mathbb{R}^2$. Furthermore, for a statement $S$, we define
        \begin{align}\label{2-1}
            1_S=
            \begin{cases}
                1\quad\text{if}~S~\text{is true},\\
                0\quad\text{if}~S~\text{is false}.
            \end{cases}
        \end{align}

        For an open set $V\subset\mathbb{R}^2$, in any boundary integral like $\int_{\partial V}\textbf{f}\cdot \textbf{n}~dS$, the vector $\textbf{n}$ is taken as the outer normal of the open set $V$ in that integral. 
        Furthermore, we define 
        \begin{equation}\label{1-7}
            l=\int_D\textbf{v}\cdot\nabla\left(\omega*N+\sum_{i=1}^{k}N(x-x_i)+\frac{\Omega}{2}|x|^2\right)dx.
        \end{equation}
        We will show that $\textbf{v}$ is divergence-free implies $l=0$ (see \textbf{Step.1} in the proof of \textbf{Theorem} \ref{Th1-4}), on the other hand, $D$ is not radial implies $l\neq 0$ for suitable divergence-free vector field $\textbf{v}$ (see \textbf{Step.2} in the proof of \textbf{Theorem} \ref{Th1-4})).
        
        In \eqref{1-7}, we need to seek a divergence-free vector field $\textbf{v}$ to derive a contradiction. Since each connected component of the background vorticity can be either simply connected, multiply connected, or possess an arbitrary number of holes, a suitable function should satisfy a Poisson equation within each connected component, with appropriate constant values prescribed on the corresponding boundaries. Hence, Lemma 2.6 and Proposition 2.8 in \cite{Serr} precisely provide the necessary framework. We state them as follows:
        \begin{lemma}\label{lem 2-1}
            Let $D\subset\mathbb{R}^2$ be a bounded domain with $C^{1,\alpha}$ boundary. Assume that $D$ has $n$ holes with $n\geq 0$, and let $h_1,...,h_n\subset\mathbb{R}^2$ denote the n holes of $D$ (each $h_i$ is a simply connected bounded domain) .The outer boundary of $D$, which we denote by $\partial D_0$. 
            
            Then there exist positive constants $\left\{c_i\right\}_{i=1}^n$ such that the solution $p$ to the Poisson equation
            \begin{align*}
                \lset{
                    \begin{array}{ll}
                        \Delta p=-2&\text{in}\ D,\\
                        p=c_i&\text{on}\ \partial h_i \ ,\ \ \text{for}\ i=1,...,n,\\
                        p=0&\text{on}\ \partial D_0,
                    \end{array}
                }
            \end{align*}
        satisfies
        \begin{equation*}
            \int_{\partial h_i}\nabla p\cdot\textbf{n}~dS=-2|h_i|\ \ \text{for}\  i=1,...,n,
        \end{equation*}
        where $|\cdot|$ denotes the 2D Lebesgue measure. Furthermore, the following two estimates hold:
        \begin{equation}\label{2-2}
            \sup_{\overline{D}}p\leq\frac{|D|}{2\pi}
        \end{equation}
        and
        \begin{equation}
            \int_Dp(x)\,dx\leq\frac{|D|^2}{4\pi}.
        \end{equation}
        For each of the two inequalities above, the equality is achieved if and only if $D$ is either a disk or an annulus.
        \end{lemma}
        \begin{remark}
            We will choose $v=-\nabla(|x|^2/2+p)$ in the proof of Theorem \ref{1-4} to show the symmetry of $D$.
        \end{remark}
        In the case of $\Omega<0$, we need the following lemma (see \cite{Serr}):
        \begin{lemma}\label{lem 2-3}
            Let $D$ be a bounded domain with $C^{1,\alpha}$ boundary, and let $p$ be as in Lemma \ref{2-1}. Then we have
            \begin{equation}
                -\int_{D}\nabla p\cdot x\,dx=\int_{D}\left|\nabla p\right|^2\,dx\leq\int_{D}|x|^2\,dx.
            \end{equation}
            Furthermore, in the inequality, equality is achieved if and only if $D$ is a disk or annulus centered at the origin.
        \end{lemma}
        In the case of $\Omega=0$, we need the following lemma (see \cite{Serr}):
        \begin{lemma}\label{lem 2-4}
            Assume that $g\in L^\infty(\mathbb{R}^2)$ is radially symmetric about some $O\in\mathbb{R}^2$, and is compactly supported in $B_R(O)$. Then $\eta:=g*N$ satisfies the following:
            \begin{itemize}
                \item[(a)] $\eta(x)=\frac{\int_{\mathbb{R}^2}gdx}{2\pi}\ln|x-O|$ for all $x\in B^c_R(O)$.\vspace{0.25em}
                \item[(b)] if in addition we have g=0 in $B_r(O)$ for some $r\in(0,R)$, then $\eta=const$ in $B_r(O)$.
            \end{itemize}
        \end{lemma}
        The weak conception of uniformly rotation solution in Definition \ref{1-2} is not conducive to proving the symmetry of $D$, thus we need a equivalent description that is as following:
        \begin{lemma}\label{2-5}
            Assume that $\left(D,x_1,...,x_k\right)$ is a uniformly rotating patch solution with angular velocity $\Omega$ of the vortex-wave system \eqref{1-4}, then
            \begin{equation}
                1_D*N+\sum_{i=1}^{k}N(x-x_i)+\frac{\Omega}{2}|x|^2=constant~~~\text{in each connected component of} ~\partial D.
            \end{equation}
        \end{lemma}
        \begin{proof}
        For any $\phi\in C_0^\infty(\mathbb{R}^2)$, we have
        \begin{equation*}
            \begin{split}
                0=&\int_{D}\nabla^\perp\left(1_D*N+\sum_{i=1}^{k}N(x-x_i)+\frac{\Omega}{2}|x|^2\right)\cdot\nabla\phi\,dx\\
                =&\int_{\partial D}\phi\nabla^\perp\left(1_D*N+\sum_{i=1}^{k}N(x-x_i)+\frac{\Omega}{2}|x|^2\right)\cdot\textbf{n}\,dS.
            \end{split}
        \end{equation*}
        Since $\phi$ is arbitrary, we have $\nabla^\perp\left(1_D*N+\sum_{i=1}^{k}N(x-x_i)+\frac{\Omega}{2}|x|^2\right)\cdot\textbf{n}=0$ on $\partial D$, which implies the conclusion.
        \end{proof}

   \subsection{	\textbf{Proof of Theorem \ref{Th1-4}}}
        Assume that $D=\cup_{i=1}^nD_i$ and each $D_i$ is a bounded domain which can be non-simply connected, and then let $h_{i1},...,h_{iM_i}\subset\mathbb{R}^2$ denote the $M_i$ holes of $D_i$ (each $h_{ij}$ is a bounded simply connected domain). Note that $\partial D_i$ has $M_i+1$ connected components: they include the outer boundary of $D_i$, which we denote by $\partial D^0_i$, and the inner boundaries $\partial h_{ij}$ for $j =1,...,M_i$. We orient $\partial h_{ij}$ the opposite way to $\partial D^0_i$. 
        Recalling that we will use $l$ (defined in \eqref{1-7}) to prove the symmetry of $D$, in the following, we will use two different ways to compute $l$. Firstly, we give the precise definition of $v$ in the integral of \eqref{1-7}.

        Let $p_i:D_i\to\mathbb{R}$ be defined as in \eqref{lem 2-1}, that is, $p_i$ satisfies
        \begin{align*}
            \lset{
                \begin{array}{ll}
                    \Delta p_i=-2&\text{in}\ D_i\ ,\vspace{0.25em}\\
                    p_i=c_{ij}&\text{on}\ \partial h_{ij}\ ,\quad \text{for}\ j=1,...,M_i,\vspace{0.25em}\\
                    p_i=0&\text{on}\ \partial D^0_i\ ,
                \end{array}
            }
        \end{align*}
        where $c_{ij}$ is chosen so that
        \begin{equation}\label{3-1}
            \int_{\partial h_{ij}}\nabla p_i\cdot\textbf{n}~dS=-2|h_{ij}|.
        \end{equation}

        We then define $\phi :D\to\mathbb{R}$, such that in each $D_i$ we have $\phi=\phi_i:=|x|^2/2+p_i$, and $v=\nabla\phi$. Define
        \begin{equation*}
            f_\Omega(x)=1_{D}*N(x)+\sum_{i=1}^{k}N(x-x_i)+\frac{\Omega}{2}|x|^2.
        \end{equation*}

        \textbf{Step.1}(the first way to compute $l$)
        \begin{align*}
            l=&\sum_{i=1}^{n}\int_{D_i}\nabla\phi_i\cdot\nabla f_\Omega \,dx\\
            =&\sum_{i=1}^{n}\int_{D_i}\nabla\cdot\left(f_\Omega\nabla\phi_i\right)\,dx\\
            =&\sum_{i=1}^{n}\int_{\partial D_i}f_\Omega\nabla\phi_i\cdot\textbf{n}\,dS\\
            =&\sum_{i=1}^{n}\left(f_\Omega\Big|_{\partial D_i^0}\int_{\partial D_i^0}\nabla\phi_i\cdot\textbf{n}\,dS-\sum_{j=1}^{M_i}f_\Omega\Big|_{\partial h_{ij}}\int_{\partial h_{ij}}\nabla\phi_i\cdot\textbf{n}\,dS\right)\\
            =&\sum_{i=1}^{n}\sum_{j=1}^{M_i}\left(f_\Omega\Big|_{\partial D^0_i}-f_\Omega\Big|_{\partial h_{ij}}\right)\int_{\partial h_{ij}}\nabla\phi_i\cdot\textbf{n}\,dS\\
            =&0,
        \end{align*}
        where we have used the fact that $\Delta\phi_i=0$, $f_\Omega$ is constant on each connected component of $\partial D$ (see Lemma \ref{2-5}) and $\int_{\partial{h_{ij}}}\nabla\phi_i\cdot\textbf{n}\,dS$=0 (see \eqref{3-1}).

        \textbf{Step.2}(the second way to compute $l$)
        \begin{align*}
            l=&\sum_{i,j=1}^{n}\int_{D_i}\left(x+\nabla p_i\right)\cdot\nabla\left(1_{D_j}*N\right)\,dx+\sum_{i=1}^{n}\sum_{m=1}^{k}\int_{D_i}\left(x+\nabla p_i\right)\cdot\nabla N(x-x_m)\,dx\\
            &+\Omega\sum_{i=1}^{n}\int_{D_i}\left(x+\nabla p_i\right)\cdot x\,dx\\
            =&:l_1+l_2+l_3.
        \end{align*}
        For $l_1$, we have
        \begin{equation}\label{3-2}
            \begin{split}
                l_1=&\sum_{i,j}\int_{D_i}x\cdot\nabla\left(1_{D_j}*N\right)\,dx+\sum_{i,j}\int_{D_i}\nabla p_i\cdot\nabla\left(1_{D_j}*N\right)\,dx\\
                =&:l_1^1+l_1^2.
            \end{split}
        \end{equation}
        For $l_1^1$, we have
        \begin{equation}\label{3-3}
            \begin{split}
                l_1^1=&-\frac{1}{2\pi}\sum_{i,j}\int_{D_i}\int_{D_j}\frac{x\cdot\left(x-y\right)}{|x-y|^2}\,dy\,dx\\
                =&-\frac{1}{4\pi}\left(\sum_{i,j}\int_{D_i}\int_{D_j}\frac{x\cdot\left(x-y\right)}{|x-y|^2}\,dy\,dx+\sum_{i,j}\int_{D_j}\int_{D_i}\frac{y\cdot\left(y-x\right)}{|x-y|^2}\,dx\,dy\right)\\
                =&-\frac{1}{4\pi}\sum_{i,j}|D_i||D_j|.
            \end{split}
        \end{equation}
        For $i\neq j$, we denote $j\prec i$ if $D_j$ is contained in a hole of $D_i$ (we also denote that hole by $h_{ic(j)}$). Then, for $l_1^2$, we have
        \begin{equation}\label{3-4}
            \begin{split}
                l_1^2=&\sum_{i,j}\left(-\int_{D_i}p_i\Delta\left(1_{D_j}*N\right)\,dx+\int_{\partial D_i}p_i\nabla\left(1_{D_j}*N\right)\cdot\textbf{n}\,dS\right)\\
                =&\sum_{i=1}^{n}\int_{D_i}p_i\,dx-\sum_{i,j}\sum_{m=1}^{M_i}c_{im}\int_{\partial h_{im}}\nabla\left(1_{D_j}*N\right)\cdot\textbf{n}\,dS\\
                =&\sum_{i=1}^{n}\int_{D_i}p_i\,dx+\sum_{i\neq j,j\prec i}c_{ic(j)}|D_j|\\
                \leq&\sum_{i=1}^{n}\int_{D_i}p_i\,dx+\frac{1}{4\pi}\left(\sum_{i,j}\left(1_{j\prec i}+1_{i\prec j}\right)|D_i||D_j|\right)\\
                =&\sum_{i=1}^{n}\int_{D_i}p_i\,dx+\frac{1}{4\pi}\sum_{j\prec i~or~i\prec j}|D_i||D_j|.
            \end{split}
        \end{equation}
        From \eqref{3-2}, \eqref{3-3} and \eqref{3-4}, we obtain
        \begin{equation}\label{3-5}
            l_1\leq\sum_{i=1}^{n}\left(\int_{D_i}p_i\,dx-\frac{|D_i|^2}{4\pi}\right)-\frac{1}{4\pi}\sum_{\substack{i\neq j\\j\nprec i~and~i\nprec j}}|D_i||D_j|.
        \end{equation}
        For $l_2$, we write
        \begin{equation}\label{3-6}
            \begin{split}
                l_2=&\sum_{i=1}^{n}\sum_{m=1}^{k}\int_{D_i}x\cdot\nabla N(x-x_m)\,dx+\sum_{i=1}^{n}\sum_{m=1}^{k}\int_{D_i}\nabla p_i\cdot\nabla N(x-x_m)\,dx\\
                =&:l_2^1+l_2^2.
            \end{split}
        \end{equation}
        For $l_2^1$, we have
        \begin{align}
                l_2^1=&-\frac{1}{2\pi}\sum_{i=1}^{n}\sum_{m=1}^{k}\int_{D_i}\frac{x\cdot\left(x-x_m\right)}{|x-x_m|^2}\,dx \notag \\
                =&-\frac{1}{2\pi}\sum_{i=1}^{n}\sum_{m=1}^{k}\int_{D_i}\frac{|x-x_m|^2+x_m\cdot\left(x-x_m\right)}{|x-x_m|^2}\,dx\notag\\
                =&-\frac{k|D|}{2\pi}-\frac{1}{2\pi}\sum_{m=1}^{k}\int_{D}\frac{x_m\cdot(x-x_m)}{|x-x_m|^2}\label{3-7}\\
                =&-\frac{k|D|}{2\pi}+\sum_{m=1}^{k}\left(\Omega|x_m|^2-\sum_{i=1,i\neq m}^{k}\frac{x_m\cdot\left(x_m-x_i\right)}{2\pi|x_m-x_i|^2}\right)\notag\\
                =&-\frac{k|D|}{2\pi}+\Omega\sum_{m=1}^{k}|x_m|^2-\frac{k(k-1)}{4\pi},\notag
        \end{align}
        where we have used the second equation of \eqref{1-6}. For $l_2^2$, we have
        \begin{equation}\label{3-8}
            \begin{split}
                l_2^2=&-\sum_{i=1}^{n}\sum_{m=1}^{k}\sum_{p=1}^{M_i}c_{ip}\int_{\partial h_{ip}}\frac{\partial N(x-x_m)}{\partial \textbf{n}}\,dS\\
                =&\sum_{i=1}^{n}\sum_{m=1}^{k}\sum_{p=1}^{M_i}c_{ip}1_{x_m\in h_{ip}}.
            \end{split}
        \end{equation}
        Combining \eqref{3-6}, \eqref{3-7} and \eqref{3-8}, we obtain
        \begin{equation}\label{3-9}
            l_2=-\frac{k|D|}{2\pi}+\Omega\sum_{m=1}^{k}|x_m|^2-\frac{k(k-1)}{4\pi}+\sum_{i=1}^{n}\sum_{m=1}^{k}\sum_{p=1}^{M_i}c_{ip}1_{x_m\in h_{ip}}.
        \end{equation}
        For $l_3$, by Lemma \ref{lem 2-3}, we have
        \begin{equation}\label{3-10}
            l_3=\Omega\sum_{i=1}^{n}\left(\int_{D_i}|x|^2\,dx-\int_{D_i}|\nabla p_i|^2\,dx\right).
        \end{equation}
        From \eqref{3-5}, \eqref{3-9} and \eqref{3-10}, we obtain
        \begin{equation}\label{3-11}
            \begin{split}
                l\leq&\sum_{i=1}^{n}\left(\int_{D_i}p_i\,dx-\frac{|D_i|^2}{4\pi}\right)+\left(-\frac{1}{4\pi}\sum_{\substack{i\neq j\\j\nprec i~and~i\nprec j}}|D_i||D_j|\right)\\
                &+\left(\sum_{i=1}^{n}\sum_{m=1}^{k}\sum_{p=1}^{M_i}c_{ip}1_{x_m\in h_{ip}}-\frac{k|D|}{2\pi}\right)\\
                &+\Omega\left(\sum_{i=1}^{n}\left(\int_{D_i}|x|^2\,dx-\int_{D_i}|\nabla p_i|^2\,dx\right)+\sum_{m=1}^{k}|x_m|^2\right)+\left(-\frac{k(k-1)}{4\pi}\right).
            \end{split}
        \end{equation}
        By Lemma \ref{lem 2-1}, Lemma \ref{lem 2-3} and the assumption $\Omega<0$, we can easily deduce that each term in the right hand side of \eqref{3-11} is non-positive. Combining this fact and \textbf{Step.1}, we conclude that each term in the right hand side of \eqref{3-11} vanishes. The fact that the fourth term is vanishing implies that each $D_i$ is either a disk or an annulus centered at origin, $k=1$ and $x_1=0$ by Lemma \ref{lem 2-3}. By the assumption that $x_1\in\mathbb{R}^2\setminus\overline{D}$, each $D_i$ must be an annulus, the proof is thus completed.

        \qed

    \subsection{\textbf{Proof of Theorem \ref{Th1-5}}}
        Note that the computations in the proof of Theorem \ref{Th1-4} still hold for the case of $\Omega=0$. In this case, in fact, we can conclude that(see \eqref{3-11})
        \begin{itemize}
            \item[(1)] Each $D_i$ is either a disk or an annulus(see the first term in \eqref{3-11} and Lemma \ref{lem 2-1}).
            \item[(2)] $D_i$ and $D_j$ are nested for $i\neq j$(see the second term in \eqref{3-11}).
            \item[(3)] $k=1$(see the fifth term in \eqref{3-11}).
            \item[(4)] $x_1$ must belong to one hole of all $D_i$ for $i=1,...,n$(see the third term in \eqref{3-11} and the fact that $c_{ip}\leq |D_i|/2\pi$ by \eqref{2-2}), hence each $D_i$ must be an annulus and $x_1$ belongs to the hole of the smallest annulus by (1) and (2).
        \end{itemize}

        To finish the proof, we only need to check that $x_1$ is the center of the all annulus $D_i$. In the following, we reorder the $D_i$ such that $D_i$ is contained in the hole of $D_{i+1}$. We also denote the inner boundary and outer boundary of $D_i$ by $\Gamma_i^{in}$ and $\Gamma_i^{out}$ for $i=1,...,n$ and denote the center of $D_i$ by $O_i$. Next, we show $O_i=x_1$ for $i=1,...,n$.

        Firstly, we show $O_1=x_1$. For $x\in \Gamma_1^{out}$, we split $f$ into
        \begin{equation*}
            f=1_{D_1}*N+\sum_{i=2}^{n}1_{D_i}*N+N(x-x_1).
        \end{equation*}
        Lemma \ref{lem 2-4}(a) yields that $1_{D_1}*N(x)=-|D_1|/2\pi\ln|x-O_1|=constant$ on $\Gamma_1^{out}$. In the second term, for each $i\geq 2$, since each $D_i$ is an annulus with $\Gamma_1^{out}$ in its hole, Lemma \ref{lem 2-4}(b) gives that the second term is constant along $\Gamma_1^{out}$. So we deduce that $N(x-x_1)$ is constant along $\Gamma_1^{out}$, which implies that $O_1=x_1$.

        Secondly, we assume for $i\leq k$, $D_i$ are known to be concentric about $x_1$. To show that $D_{k+1}$ is also centered at $x_1$, for $x\in\Gamma_{k+1}^{out}$, we decompose $f$ into
        \begin{equation*}
            \begin{split}
                f=&\sum_{i=1}^{k}1_{D_j}*N+1_{D_{k+1}}*N+\sum_{i\geq k+1}^{n}1_{D_i}*N+N(x-x_1)\\
                =&-\frac{\sum_{i=1}^{k}|D_i|}{2\pi}\ln|x-x_1|-\frac{1}{2\pi}\ln|x-x_1|+constant\\
                =&-\frac{1+\sum_{i=1}^{k}|D_i|}{2\pi}\ln|x-x_1|+constant,
            \end{split}
        \end{equation*}
        which implies that $\ln|x-x_1|$ is constant along $\Gamma_{k+1}^{out}$, so we get $O_{k+1}=x_1$. The proof is thus completed.

        \qed

\section{Non-radial symmetric uniformly rotating patch solutions for \texorpdfstring{$\Om>0$}{}}\label{sec-4}
    This section is devoted to the proof of Theorem \ref{main theorem 2}, which shows the existence of a family of non-radial symmetric uniformly rotating patch solutions. During the rest of this paper, we will fix two real numbers $0<a_1<a_2$ and only focus on the vortex wave systems \eqref{1-4} in the case of $k=1$.
    Therefore we can rewrite \eqref{1-4} as 
    \begin{equation}\label{VWE k-1}
        \begin{cases}
            \partial_t \omega+\mathbf{u}\cdot\nabla \omega=0,\\
            \frac{d x_1(t)}{dt}=\nabla^\perp N*\omega(x_1(t),t),\\
            \mathbf{u}(x,t)=\nabla^\perp \left\{N*\omega(x,t)+ N(x-x_1(t))\right\},
        \end{cases}   
    \end{equation}
    where $N(x)=-\frac{1}{2\pi}\ln|x|$ is the fundamental solution of the Laplacian $ -\Delta $ in $ \mathbb{R}^2 $ and $(a,b)^\perp=(b,-a)$ for any vector $(a,b)\in\R^2$. We will split this section into several subsections for better readability.

    \subsection{Notations and preliminaries}
        Before getting into the details of the contour dynamics equation, we setup some notations to be constantly used thereafter, along with several theorems essential for our analysis.
    
        In the sequel, we will also agree the identification $\mathbb{C} \approx \R^2$ and denote the scalar product of two complex numbers $c_1=p_1+iq_1$ and $c_2=p_2+iq_2$ by 
        \begin{equation*}
            c_1\cdot c_2={\rm Re}(c_1\overline{c_2})=p_1p_2+q_1q_2.
        \end{equation*}
        Notice that $ic_1=-(p_1,q_1)^\perp$ and denote $\T=\R /2\pi\Z$.

        Next we recall an important theorem of bifurcation theory, established by Crandall and Rabinowitz in \cite{CR-Bifurcation}, which will play a central role in the proof of Theorem \ref{main theorem 2}.
        \begin{theorem}[Crandall-Rabinowitz]\label{CR-Fir}
            Let $X$ and $Y$ be two Banach spaces. Let $V\subset X$ be a neighborhood of $0$ and $F$ be a functional such that 
            \begin{equation*}
                F:\R\times V\to Y.
            \end{equation*}
            Assume that $F$ enjoys the following properties:
            \begin{itemize}
                \item[(1)] Existence of trivial branch:
                        \begin{equation*}
                            F(\Om,0)=0,\quad for\ all\ \Om\in\R.
                        \end{equation*}
                \item[(2)] Regularity: $F$ is regular in the sense that $\partial_\Om F, d_x F$ and $\partial_\Om d_xF$ exist and are continuous. 
                \item[(3)] Fredholm property: The kernel $\ker (d_xF(0,0))$ is of dimension one, i.e., there is $x_0\in V$ such that 
                        \begin{equation*}
                            \ker (d_x F(0,0))=\Span\{x_0\},
                        \end{equation*}  
                        and the $\Ran (d_x F(0,0))$ is closed and of a co-dimension one.
                \item[(4)] Transversality:
                        \begin{equation*}
                            \partial_\Om d_xF(0,0)[x_0]\notin \Ran(d_xF(0,0)).
                        \end{equation*}
            \end{itemize}
            Then, denoting $\chi$ any complement of $\ker(d_xF(0,0))$ in $X$, there exists a neighborhood $U$ of $(0,0)$, an interval $(-a,a)$, for some $a>0$ and continuous functions
            \begin{equation*}
                \psi:(-a,a)\to\R\quad {\rm and}\quad\phi:(-a,a)\to\chi,
            \end{equation*}
            such that 
            \begin{equation*}
                \psi(0)=0,\quad \phi(0)=0
            \end{equation*}
            and
            \begin{equation*}
                \left\{(\Om,x)\in U:F(\Om,x)=0\right\}=\left\{(\psi(s),sx_0+s\phi(s)):|s|<a\right\}\cup\left\{(\Om,0)\in U\right\}.
            \end{equation*}
        \end{theorem}
        
        We conclude this subsection with a simple but powerful identity (see \cite{HHE-identity}).
        \begin{lemma}\label{identity-ln}
            For $x\in(0,\infty)$, there holds
            \begin{equation}
                \frac{1}{2\pi}\int_0^{2\pi}\ln|1-xe^{i\eta}|\cos (n\eta)\,d\eta=-\frac{\min\{x^n,x^{-n}\}}{2n},\quad for\ all\ n\in \N^*.
            \end{equation}
        \end{lemma}

    \subsection{Contour dynamics} 
        Our first task is to find some continuous functions satisfied the first assumption in Theorem \ref{CR-Fir}.

        We look for a uniformly rigid rotating solution $(\omega(x,t),x_1(t))$ of \eqref{VWE k-1} in the form 
        \begin{equation}
            \lset{
                \begin{array}{ll}
                    \omega(\cdot,t)=1_{D_{2,t}\setminus \overline{D_{1,t}}}\\
                    x_1(t)=e^{i\Om t}x_1
                \end{array}
            }
        \end{equation}
        for $t>0$, where $D_{l,t}=e^{i\Om t}D_l$ and the boundary of domain $D_{l}$ is close to a circle $\{x\in\R^2: |x|<a_l\}$ in a sense to be precised in a moment, for $l=1,2$. To this end, we consider the parametrizations of boundaries $z_l(\theta,t):\T\to\partial D_{l,t}$ with the following ansatz, for $l=1,2$
        \begin{equation}\label{parametrization of z1t z2t}
            z_l(\theta,t)=e^{i\Om t}z_l(\theta),
        \end{equation}
        where $(z_1(\theta),z_2(\theta))_{\theta\in\T}$ is a parametrization of the initial patch $1_{D_2\setminus D_1}$ with
        \begin{equation}\label{parametrization of z1 z2}
            \begin{split}
                z_l:\T&\to\partial D_l\\
                \theta&\to R_l(\theta)e^{i\theta}:=\sqrt{a_l^2+2r_l(\theta)}e^{i\theta}.    
            \end{split}
        \end{equation}
    
        In the following, we always assume that $x_1=(x_1^1,0)\in\R^2$.
        Subsequently, we shall establish an equivalent reformulation of \eqref{VWE k-1} in the case of uniformly rotating solutions.
        
        \begin{lemma}[Contour systems]\label{Countor equation lemma}
            Let $\Om\in\R$ and $r_1,r_2$ be two even functions of $\theta$, i.e.
            \begin{equation}
                r_1(-\theta)=r_1(\theta),\quad r_2(-\theta)=r_2(\theta),\quad {\rm for\ all}\ \theta\in \T.
            \end{equation}
            The radial deformation defined through \eqref{parametrization of z1 z2} gives rise to a uniformly rotating, at angular velocity $\Om$, solution to the \eqref{VWE k-1} if and only if it satisfies the nonlinear contour systems
            \begin{equation*}
                (F_1,F_2,G)(\Om,r_1,r_2,x_1^1)=0,
            \end{equation*}
            where, for any $\theta\in\T$, we denote
            \begin{equation}\label{Def of (F_1,F_2,G)}
                \begin{split}
                        &F_1(\Om,r_1,r_2,x_1^1)= \Om r_1'(\theta)-\frac{1}{2\pi}\partial_\theta\left\{\int_0^{2\pi}\int_{R_1(\eta)}^{R_2(\eta)}\ln|R_1(\theta)e^{i\theta}-\rho e^{i\eta}|\rho \,d\rho\,d\eta\right.\\
                        &\left.\quad\quad\quad\quad\quad\quad\quad\quad\quad\quad\quad\quad\quad\qquad+ \ln|R_1(\theta)e^{i\theta}-(x_1^1,0) |\right\},\\
                        &F_2(\Om,r_1,r_2,x_1^1)=\Om r_2'(\theta)-\frac{1}{2\pi}\partial_\theta\left\{\int_0^{2\pi}\int_{R_1(\eta)}^{R_2(\eta)}\ln|R_2(\theta)e^{i\theta}-\rho e^{i\eta}|\rho \,d\rho\,d\eta\right.\\
                        &\left.\quad\quad\quad\quad\quad\quad\quad\quad\quad\quad\quad\quad\quad\qquad+\ln|R_2(\theta)e^{i\theta}-(x_1^1,0) |\right\},\\
                        &G(\Om,r_1,r_2,x_1^1)\ =\Om x_1^1-\frac{1}{2\pi}\int_0^{2\pi}\int_{R_1(\eta)}^{R_2(\eta)}\frac{x_1^1-\rho \cos\eta}{(x_1^1-\rho \cos\eta)^2+(\rho\sin\eta)^2}\rho \,d\rho\,d\eta.
                \end{split}
            \end{equation}
            Furthermore, the trivial deformation, corresponding to $(r_1,r_2,x_1^1)=(0,0,0)$ is a solution to the preceding contour systems for all $\Om\in\R$.
        \end{lemma}
        \begin{proof}
            Let us first consider, for any time $t\geq 0$ and $l=1,2$, the parametrizations $z_l(\cdot,t):\T\to\partial D_{l,t}$ of the boundaries. According to \cite[page 174]{Hmidi-boundary}, we find, for any $\theta\in \T$ and $t\geq 0$, that
            \begin{equation}\label{4.5}
                \partial_t z_l(\theta,t)\cdot \mathbf{n}_l(z_l(\theta,t),t)=\mathbf{u}(z_l(\theta,t),t)\cdot \mathbf{n}_l,\quad {\rm for }\ l=1,2,
            \end{equation}
            where $\mathbf{n}_l=-i \frac{\partial z_l}{\partial \theta}$ is the normal vector to the boundary $\partial D_{l,t}$ at the point $z_l(\theta,t)$.
            
            On the one hand, the parametrization $z_l(\theta,t)=R_l(\theta)e^{i\theta}e^{i\Om t}$ shows that 
            \begin{align*}
                \partial_t z_l(\theta,t)\cdot \mathbf{n}_l(z_l(\theta,t),t)&=i\Om R_l(\theta)e^{i\theta}e^{i\Om t}\cdot(-i\partial_\theta(R_l(\theta)e^{i\theta}e^{i\Om t}))\\
                &=-\Om r'_l(\theta).
            \end{align*}
            On the other hand, due to \eqref{VWE k-1}, we have
            \begin{align*}
                \mathbf{u}(z_l(\theta,t),t)\cdot \mathbf{n}_l&=-\nabla^\perp \left\{N*\omega(z_l(\theta,t),t)+ N(z_l(\theta,t)-x_1(t))\right\}\cdot \left\{i \partial_\theta z_l\right\}\\
                &=i\nabla\left\{N*\omega (z_l(\theta,t),t)+N(z_l(\theta,t)-x_1(t)) \right\}\cdot \left\{i\partial_\theta z_l\right\}\\
                &=\partial_\theta\left\{N*\omega (z_l(\theta,t),t)+N(z_l(\theta,t)-x_1(t))\right\}.
            \end{align*}
            Thus \eqref{4.5} can be written as
            \begin{equation}
                \Om r_l'(\theta)+\partial_\theta\left\{N*\omega (z_l(\theta,t),t)+N(z_l(\theta,t)-x_1(t))\right\}=0,\quad {\rm for}\ l=1,2.
            \end{equation}
            
            Thanks to the uniformly rotating property of $\omega(x,t)$ and $x_1(t)$, one sees that the time-variable can be discarded from the equations.
            In fact, we obtain that, for $l=1,2$,
            \begin{equation*}
                \ln|z_l(\theta,t)-x_1(t)|=\ln|R_l(\theta)e^{i\theta}-(x_1^1,0)|,
            \end{equation*}
            and 
            \begin{align*}
                N*\omega(z_l(\theta,t),t)&=-\frac{1}{2\pi}\int_{D_0 e^{i\Om t}}\ln|R_l(\theta)e^{i\theta}e^{i\Om t}-y|\,dy\\
                &=-\frac{1}{2\pi}\int_{D_0}\ln|R_l(\theta)e^{i\theta}e^{i\Om t}-ye^{i\Om t}|\,dy\\
                &=-\frac{1}{2\pi}\int_{\R^2}\ln|R_l(\theta)e^{i\theta}-y|1_{D_0}\,dy\\
                &=N*\omega(R_l(\theta)e^{i\theta},0).
            \end{align*}
            
            As for $x_1(t)$, let us consider the second component of dynamic equation of the point vortex. Denote 
            \[
                A(r_1,r_2,x_1^1)=\frac{1}{2\pi}\int_0^{2\pi}\int_{R_1(\eta)}^{R_2(\eta)}\frac{-\rho \sin\eta}{|(x_1^1,0)-\rho e^{i\eta}|^2}\rho \,d\rho\,d\eta.
            \]
            It follows from the even symmetry of $r_1,r_2$ that
            \begin{equation*}
                \begin{split}
                    A(r_1,r_2,x_1^1)=&-\frac{1}{2\pi}\int_0^{-2\pi}\int_{R_1(-\eta)}^{R_2(-\eta)}\frac{-\rho \sin(-\eta)}{|(x_1^1,0)-\rho e^{-i\eta}|^2}\rho \,d\rho\,d\eta\\
                    =&-\frac{1}{2\pi}\int_0^{2\pi}\int_{R_1(\eta)}^{R_2(\eta)}\frac{-\rho \sin\eta}{|(x_1^1,0)-\rho e^{i\eta}|^2}\rho \,d\rho\,d\eta\\
                    =&-A(r_1,r_2,x_1^1).
                \end{split}
            \end{equation*}
            Then we find that $A(r_1,r_2,x_1^1)=0$, as soon as $r_1$ and $r_2$ are even. So we need only focus on the first component of dynamic equation of the point vortex.
            Therefore, we have 
            \begin{align*}
                &\Om r_1'(\theta)=\frac{1}{2\pi}\partial_\theta\left(\int_0^{2\pi}\int_{R_1(\eta)}^{R_2(\eta)}\ln|R_1(\theta)e^{i\theta}-\rho e^{i\eta}|\rho \,d\rho\,d\eta+\ln|R_1(\theta)e^{i\theta}-(x_1^1,0) |\right),\\
                &\Om r_2'(\theta)=\frac{1}{2\pi}\partial_\theta\left(\int_0^{2\pi}\int_{R_1(\eta)}^{R_2(\eta)}\ln|R_2(\theta)e^{i\theta}-\rho e^{i\eta}|\rho \,d\rho\,d\eta+\ln|R_2(\theta)e^{i\theta}-(x_1^1,0) |\right),\\
                &\Om x_1^1=\frac{1}{2\pi}\int_0^{2\pi}\int_{R_1(\eta)}^{R_2(\eta)}\frac{x_1^1-\rho \cos\eta}{(x_1^1-\rho\cos\eta)^2+(\rho \sin\eta)^2}\rho \,d\rho\,d\eta.
            \end{align*}
            This derives the nonlinear contour systems as it is claimed in Lemma \ref{Countor equation lemma}, above.

            Now it is easy to ascertain that $(\Om,0,0,0)$ is a zero of the functional systems $(F_1,F_2,G)$ for any $\Om\in\R.$
            In fact, followed by a change of variables, it implies that $\ln|a_le^{i\theta}|=\ln|a_l|$ and
            \begin{equation}
                \int_{0}^{2\pi}\int_{a_1}^{a_2}\ln|a_l e^{i\theta}-\rho e^{i\eta}|\rho\,d\rho\,d\eta=\int_{0}^{2\pi}\int_{a_1}^{a_2}\ln|a_l-\rho e^{i\eta}|\rho\,d\rho\,d\eta,
            \end{equation}
            thereby yielding that
            \begin{equation}
                \partial_\theta \left\{\int_{0}^{2\pi}\int_{a_1}^{a_2}\ln|a_l e^{i\theta}-\rho e^{i\eta}|\rho\,d\rho\,d\eta-\ln|a_l e^{i\theta}|\right\}=0.
            \end{equation}
            Furthermore, we obtain that
            \begin{equation}
                \int_{0}^{2\pi}\int_{a_1}^{a_2}\rho \cos\eta\,d\rho\,d\eta=0.
            \end{equation}
            The proof is thus completed.
        \end{proof}

        Next we give an simple but important symmetry property of the nonlinear functional systems $(F_1,F_2)$.
        \begin{lemma}\label{symmetry property of F1 F2 G}
            Let $\Om\in\R$, $r_1,r_2\in C^{1+\alpha}(\T)$ for some $\alpha\in(0,1)$, $x_1=(x_1^1,0)\in\R^2$ and $F_1,F_2$ be given by \eqref{Def of (F_1,F_2,G)}. If moreover $r_1,r_2$ are both even functions about $\theta$, i.e.
            \begin{equation*}
                r_l(-\theta)=r_l(\theta),\quad for\ all\ \theta\in[0,2\pi]\ and\ l=1,2,
            \end{equation*}
            then, $(F_1,F_2)(\Om,r_1,r_2,x_1^1)$ are odd functions, i.e.
            \begin{equation*}
                (F_1,F_2)(\Om,r_1,r_2,x_1^1)(-\theta)=-(F_1,F_2)(\Om,r_1,r_2,x_1^1)(\theta),\quad for\ all\ \theta\in[0,2\pi].
            \end{equation*}
        \end{lemma}
        \begin{proof}
            By direct computation, it shows that, if $r$ is an even function, then $r'$ is odd. Therefore, we can restrict our focus on showing that the following functions
            \begin{gather*}
                g_1^{j,l}(\theta):=\int_0^{2\pi}\int_0^{R_l(\eta)}\ln|R_j(\theta) e^{i\theta}-\rho e^{i\eta}|\rho\,d\rho\,d\eta,\quad{\rm  for\ all}\ j,l\in\{1,2\},
            \end{gather*}
            and 
            $$g_2^l(\theta):=\theta\mapsto \ln|R_l(\theta)e^{i\theta}-(x_1^1,0)|,\quad {\rm for\ all}\ l=1,2$$
            are even, as soon as $r_1$ and $r_2$ are even. 

            In fact, $g_1^{j,l}$ is even followed by the change of variables $\eta\to-\eta$. As for $g_2^l$, the fact $x_1=(x_1^1,0)$ implies that $|R_l(\theta)e^{i\theta}-(x_1^1,0)|=\sqrt{(R_l(\theta)\cos\theta-x_1^1)^2+(R_l(\theta)\sin\theta)^2}$.
            The proof of Lemma \ref{symmetry property of F1 F2 G} is completed.
        \end{proof}

    \subsection{Regularity analysis}
        In this section, we first discuss the Banach spaces involved in our application of the Crandall-Rabinowitz theorem and then check the regularity properties of the nonlinear contour systems $(F_1,F_2,G)$ introduced in \eqref{Def of (F_1,F_2,G)}, whose zeros decided the uniformly rotating solutions of vortex wave systems \eqref{VWE k-1}.
        
        We define, for a real number $\alpha\in(0,1)$, the function spaces $X^\alpha$ and $Y^\alpha$ by
        \begin{equation}\label{Xa-space}
            X^\alpha=\left\{f\in C^{1+\alpha}(\T):f(\theta)=\sum_{n=1}^\infty f_n\cos (n\theta),\ f_n\in\R,\ \theta\in\T\right\}
        \end{equation}
        and
        \begin{equation}\label{Ya-Space}
            Y^\alpha=\left\{g\in C^{\alpha}(\T):g(\theta)=\sum_{n=1}^\infty g_n\sin (n\theta),\ g_n\in\R,\ \theta\in\T\right\}
        \end{equation}
        and equipped with the usual $C^{1+\alpha}$ and $C^\alpha$ norms, respectively. 
        
        Furthermore, for a given $\varepsilon\in\left(0,\min\{\frac{a_2-a_1}{3},\frac{a_1}{3}\}\right),$ we define
        \begin{equation}\label{Bae-Ball}
            B_{\varepsilon}^\alpha=\left\{f\in X^\alpha: \|f\|_{C^{1+\alpha}}<\varepsilon\right\},
        \end{equation}
        and 
        \begin{equation}\label{B-Ball}
            B_\varepsilon=\left\{x_1^1\in\R:|x_1^1|<\varepsilon\right\}.
        \end{equation}
        
        The following proposition summarizes the regularity properties of the nonlinear functional systems $(F_1,F_2,G)$.
        \begin{proposition}[Regularity]\label{Regularity of F1 F2 G}
            Let $0<a_1<a_2$ and $\alpha\in(0,1)$. There exists a small parameter $\varepsilon\in(0,\min\{a_1/3,(a_2-a_1)/3\})$ such that for the functional
            \begin{equation}
                (F_1,F_2,G):\R\times B_\varepsilon^\alpha\times B_\varepsilon^\alpha\times B_\varepsilon\to Y^\alpha\times Y^\alpha\times \R
            \end{equation}
            is well-defined and of class $C^1$. Moreover, the partial derivative $\partial_\Om  d_{(r_1,r_2,x_1^1)}(F_1,F_2,G)$ exists in the sense that 
            \begin{equation}
                \partial_\Om d_{(r_1,r_2,x_1^1)}(F_1,F_2,G):\R\times B_\varepsilon^\alpha\times B_\varepsilon^\alpha\times B_\varepsilon\to \mathcal{L}(X^\alpha\times X^\alpha\times\R,Y^\alpha\times Y^\alpha\times \R )
            \end{equation}
            is continuous.
        \end{proposition}
        \begin{proof}
            Firstly, thanks to the previous results in Proposition 3.2 in \cite{Hmidi-2024-Lake} and the symmetry properties of $(F_1,F_2)$ proved in Lemma \ref{symmetry property of F1 F2 G}, by setting $b=1$, we know that, for $l=1,2$ and some small parameter $\varepsilon\in(0,\min\{a_1/3,(a_2-a_1)/3\})$, the function 
            \[ 
                f^1_l(\Om,r_1,r_2,x_1^1):=\Om r_l'(\theta)-\frac{1}{2\pi}\partial_\theta\left\{\int_0^{2\pi}\int_{R_1(\eta)}^{R_2(\eta)}\ln|R_l(\theta)e^{i\theta}-\rho e^{i\eta}|\rho\,d\rho\,d\eta\right\} 
            \] 
            is well-defined from $\R\times B_\varepsilon^\alpha\times B_\varepsilon^\alpha\times B_\varepsilon$ to $Y^\alpha$ and of class $C^1$. Moreover, the partial derivative $\partial_\Om d_{(r_1,r_2,x_1^1)}f^1_l$ exists in the sense that 
            \begin{equation}
                \partial_\Om d_{(r_1,r_2,x_1^1)}f^1_l:\R\times B_\varepsilon^\alpha\times B_\varepsilon^\alpha\times B_\varepsilon\to \mathcal{L}(X^\alpha\times X^\alpha\times\R,Y^\alpha)
            \end{equation}
            is continuous.

            As for the function 
            \[ 
                f_l^2(\Om,r_1,r_2,x_1^1):=\partial_\theta \{\ln|R_l(\theta)e^{i\theta}-(x_1^1,0)|\}=\frac{R_l(\theta)e^{i\theta}-(x_1^1,0)}{|R_l(\theta)e^{i\theta}-(x_1^1,0)|^2}\cdot\partial_\theta(R_l(\theta)e^{i\theta}),
            \]
            due to the result in \cite[Page 5]{ZGQ Nonlinear Methods}, it follows that $f^2_l :\R\times B_\varepsilon^\alpha\times B_\varepsilon^\alpha\times B_\varepsilon \to Y^\alpha$ is well-defined and of class $C^1$ for $l=1,2$. Furthermore, we can directly compute that $\partial_\Om d_{(r_1,r_2,x_1^1)}f_l^2(\Om,r_1,r_2,x_1^1)=0$, for all $(\Om,r_1,r_2,x_1^1)\in \R\times B_\varepsilon^\alpha\times B_\varepsilon^\alpha\times B_\varepsilon$.
            
            Now we are only left to prove the regularity of $G(\Om,r_1,r_2,x_1^1)$. In fact, due to the chosen of $\varepsilon$ and the results in \cite[Page 5]{ZGQ Nonlinear Methods}, an argument similar to the $f_l^2$ shows that the function $G(\Om,r_1,r_2,x_1^1)$ conforms to the regularity conditions specified in the proposition. 
            
            This completes the proof of Proposition \ref{Regularity of F1 F2 G}.
        \end{proof}

    \subsection{Spectral analysis}
        This part is crucial for implementing the Crandall-Rabinowitz theorem. We shall in particular compute the linearized operator $d_{(r_1,r_2,x_1^1)}(F_1,F_2,G)$ around the trivial solution and look for the values of $\Om$ associated with the nontrivial kernel. For these values of $\Om$, we shall see that the linearized operator has a one-dimensional kernel and is in fact of Fredholm type with zero index. Before giving the main
        result of this subsection, we first compute the linearized operator at any point.
        \begin{proposition}\label{differential at any point}
            Let $(F_1,F_2,G)$ be given by \eqref{Def of (F_1,F_2,G)}. Let $\varepsilon\in(0,\min\{(a_2-a_1)/3,a_1/3\})$ be a given small parameter and $\alpha\in(0,1)$ in such a way that $(F_1,F_2,G)$ are differentiable according to Proposition \ref{Regularity of F1 F2 G}. 
            Then the differential of $(F_1,F_2,G)$ at $(r_1,r_2,x_1^1)\in B_{\varepsilon}^\alpha\times B_{\varepsilon}^\alpha\times B_\varepsilon$ in the direction $[h_1,h_2,b]\in X^\alpha\times X^\alpha \times \R$ is given by
            \begin{equation}
                d_{(r_1,r_2,x_1^1)}(F_1,F_2,G)(\Om,r_1,r_2,x_1^1)[h_1,h_2,b]=I+II,
            \end{equation}
            where we set $ I=(I_1,I_2,I_3)^T$, $II=(II_1,II_2,II_3)^T$, for all $\theta\in[0,2\pi]$, 
            \begin{align*}
                I_1=&\Om h_1'+\partial_\theta\left\{\int_0^{2\pi}\int_{R_1(\eta)}^{R_2(\eta)}\nabla N(R_1(\theta)e^{i\theta}-\rho e^{i\eta})\cdot e^{i\theta}\frac{h_1(\theta)}{R_1(\theta)}\rho\,d\rho\,d\eta\right.\\
                &\quad\quad \qquad+ \nabla N(R_1(\theta) e^{i\theta}-(x_1^1,0))\cdot e^{i\theta}\frac{h_1(\theta)}{R_1(\theta)} \\
                &\quad\quad \qquad\left.+\int_{0}^{2\pi}\left[N(R_1(\theta)e^{i\theta}-R_2(\eta)e^{i\eta})h_2(\eta)-N(R_1(\theta e^{i\theta})-R_1(\eta)e^{i\eta})h_1(\eta)\right]\,d\eta \right\},\\
                I_2=&\Om h_2'+\partial_\theta\left\{\int_0^{2\pi}\int_{R_1(\eta)}^{R_2(\eta)}\nabla N(R_2(\theta)e^{i\theta}-\rho e^{i\eta})\cdot e^{i\theta}\frac{h_2(\theta)}{R_2(\theta)}\rho\,d\rho\,d\eta\right.\\
                &\quad\quad \qquad+ \nabla N(R_2(\theta) e^{i\theta}-(x_1^1,0))\cdot e^{i\theta}\frac{h_2(\theta)}{R_2(\theta)} \\
                &\quad\quad \qquad\left.-\int_{0}^{2\pi}\left[N(R_2(\theta)e^{i\theta}-R_1(\eta)e^{i\eta})h_1(\eta)-N(R_2(\theta e^{i\theta})-R_2(\eta)e^{i\eta})h_2(\eta)\right]\,d\eta \right\},\\
                I_3=&-\frac{1}{2\pi}\int_{0}^{2\pi}\frac{x_1^1-R_2(\eta)\cos\eta}{(x_1^1-R_2(\eta)\cos\eta)^2+(R_2(\eta)\sin\eta)^2}h_2(\eta)\,d\eta \\
                &+\frac{1}{2\pi}\int_0^{2\pi}\frac{x_1^1-R_1(\eta)\cos\eta}{(x_1^1-R_1(\eta)\cos\eta)^2+(R_1(\eta)\sin\eta)^2}h_1(\eta) \,d\eta
            \end{align*}
            and
            \begin{equation*}
                \begin{split}
                    II_1=&\frac{b}{2\pi}\frac{\partial_\theta (R_1(\theta)\cos\theta)}{|R_1(\theta)e^{i\theta}-(x_1^1,0)|^2}-\frac{b}{2\pi}\frac{(R_1(\theta)\cos\theta-x_1^1)\partial_\theta|R_1(\theta)e^{i\theta}-(x_1^1,0)|^2 }{|R_1(\theta)e^{i\theta}-(x_1^1,0)|^4},\\
                    II_2=&\frac{b}{2\pi}\frac{\partial_\theta (R_2(\theta)\cos\theta)}{|R_2(\theta)e^{i\theta}-(x_1^1,0)|^2}-\frac{b}{2\pi}\frac{(R_2(\theta)\cos\theta-x_1^1)\partial_\theta |R_2(\theta)e^{i\theta}-(x_1^1,0)|^2}{|R_2(\theta)e^{i\theta}-(x_1^1,0)|^4},\\
                    II_3=&\Om b-\frac{b}{2\pi}\int_0^{2\pi}\int_{R_1(\eta)}^{R_2(\eta)}\frac{(\rho\sin\eta)^2-(x_1^1-\rho\cos\eta)^2}{|(x_1^1,0)-\rho e^{i\eta}|^4}\rho\,d\rho\,d\eta.
                \end{split}
            \end{equation*}
        \end{proposition}
        \begin{proof}
            Due to the representation \eqref{Def of (F_1,F_2,G)} and regularities proved in Proposition \ref{Regularity of F1 F2 G}, for $[h_1,h_2,b]\in X^\alpha\times X^\alpha\times \R$, we can directly compute the derivatives of each term
            \begin{align}
                d_{(r_1,r_2,x_1^1)}(F_1,F_2,G)(\Om,r_1,r_2,x_1^1)[h_1,h_2,b]&=\begin{pmatrix}
                                                                            d_{r_1}F_1[h_1]+d_{r_2}F_1[h_2]+d_{x_1}F_1[b]\vspace{0.25em}\\
                                                                            d_{r_1}F_2[h_1]+d_{r_2}F_2[h_2]+d_{x_1}F_2[b]\vspace{0.25em}\\
                                                                            d_{r_1}G[h_1]+d_{r_2}G[h_2]+d_{x_1}G[b]
                                                                        \end{pmatrix},
            \end{align}
            where, for $l,j=1,2$,
            \begin{equation*}
                \begin{split}
                    d_{r_l}F_l[h_l]=&\Om h_l'+\partial_\theta\left\{(3-2l)\int_{0}^{2\pi}N(R_l(\theta)e^{i\theta}-R_l(\eta)e^{i\eta})h_l(\eta)\,d\eta\right.\\
                    &+\left.\frac{h_l(\theta)e^{i\theta}}{R_l(\theta)}\cdot\left[\nabla N(R_l(\theta)e^{i\theta}-(x_1^1,0))+\int_{0}^{2\pi}\int_{R_1(\eta)}^{R_2(\eta)}\nabla N(R_l(\theta)e^{i\theta}-\rho e^{i\eta})\rho\right]\right\},\vspace{0.25em}\\
                    d_{x_1^1}G[b]=&\Om b-\frac{b}{2\pi}\int_0^{2\pi}\int_{R_1(\eta)}^{R_2(\eta)}\frac{(\rho\sin\eta)^2-(x_1^1-\rho\cos\eta)^2}{|(x_1^1,0)-\rho e^{i\eta}|^4}\rho\,d\rho\,d\eta
                \end{split}
            \end{equation*}
            and 
            \begin{equation*}
                \begin{split}
                    &d_{r_l}F_j[h_l]=(3-2j)\partial_\theta\left(\int_{0}^{2\pi}N(R_j(\theta)e^{i\theta}-R_l(\eta)e^{i\eta})h_l(\eta)\,d\eta\right),\quad {\rm for}\ j\neq l,\\
                    &d_{x}F_l[b]=\frac{b}{2\pi}\left\{\frac{\partial_\theta\left(R_l(\theta)\cos\theta\right)}{|R_l(\theta)e^{i\theta}-(x_1^1,0)|^2}-\frac{(R_l(\theta)\cos\theta-x_1^1)\partial_\theta|R_l(\theta)e^{i\theta}-(x_1^1,0)|^2}{|R_l(\theta)e^{i\theta}-(x_1^1,0)|^4}\right\},\\
                    &d_{r_l}G[h_l]=(2l-3)\int_{0}^{2\pi}\frac{x_1^1-R_l(\eta)\cos\eta}{(x_1^1-R_l(\eta)\cos\eta)^2+(R_l(\eta)\sin\eta)^2}h_l(\eta)\,d\eta.
                \end{split}
            \end{equation*}

            Combining above calculations, we can obtain the statements we desired.
        \end{proof}
        The following lemma establishes several important properties of the linearized operator $d_{(r_1,r_2,x_1^1)}(F_1,F_2,G)$ at the equilibrium.
        \begin{lemma}\label{differential at 0,0,0}
            Let $0<a_1<a_2$, $\Om\in\R$ and $\alpha\in(0,1)$. Fix $\varepsilon>0$ as in Lemma \ref{Regularity of F1 F2 G}. If moreover,
            \begin{equation}\label{Om not equal set}
                \Om\notin \left\{\frac{a_2^2-a_1^2}{2a_2^2}+\frac{1}{2\pi a_2^2},\ \frac{1}{2\pi a_1^2},\ 0\right\},
            \end{equation}
            then, at the equilibrium, the linearized operator $d_{(r_1,r_2,x_1^1)}(F_1,F_2,G)(\Om,0,0,0)$ is of Fredholm type operator with zero index. Furthermore, for any $b\in\R$, $h_1,h_2\in X^\alpha$ taking the form
            \begin{equation}
                h_1(\theta)=\sum_{n=1}^\infty h_n^{(1)}\cos(n\theta),\quad h_2(\theta)=\sum_{n=1}^\infty h_n^{(2)}\cos(n\theta),
            \end{equation}
            for all $\theta\in\T$ and some scalars $h_n^{(1)},h_n^{(2)}\in\R$, it holds that
            \begin{equation}\label{ker d F_1 F_2 G at 0,0,0}
                \begin{split}
                    &d_{(r_1,r_2,x_1^1)}(F_1,F_2,G)(\Om,0,0,0)[h_1,h_2,b]\\
                    &=
                    \begin{pmatrix}
                            \sum_{n=1}nM_{n}(\Om,a_1,a_2)\begin{pmatrix}
                                h_n^{(1)}\\
                                h_n^{(2)}
                            \end{pmatrix}\sin(n\theta)-
                            \begin{pmatrix}
                                \frac{b\sin\theta }{2\pi a_1}\\
                                \frac{b\sin\theta }{2\pi a_2}
                            \end{pmatrix}\vspace{0.25em}\\
                            \frac{h_1^{(2)}}{2a_2}-\frac{h_1^{(1)}}{2a_1}+\Om b
                    \end{pmatrix},
                \end{split}
            \end{equation}
            where
            \begin{equation}\label{def of Mn}
                M_{n}(\Om,a_1,a_2)=
                \begin{pmatrix}
                    -\Om+\frac{1}{2n}+\frac{1}{2\pi a_1^2}&-\frac{1}{2n}(\frac{a_1}{a_2})^{n}\vspace{0.25em}\\
                    \frac{1}{2n}(\frac{a_1}{a_2})^{n}&-\Om+\frac{a_2^2-a_1^2}{2a_2^2}-\frac{1}{2n}+\frac{1}{2\pi a_2^2}.
                \end{pmatrix}.
            \end{equation}
        \end{lemma}
        \begin{proof}
            Taking $r_1=r_2=0,\ x_1^1=0$ in Proposition \ref{differential at any point}, we get that 
            \begin{equation*}
                \begin{split}
                    I_3=\frac{h_1^{(2)}}{2a_2}-\frac{h_1^{(1)}}{2a_1}\quad II_1=-\frac{b\sin\theta}{2\pi a_1},\quad II_2=-\frac{b\sin\theta}{2\pi a_2},\quad II_3=\Om b,
                \end{split}
            \end{equation*}
            and by Lemma 2.2 in \cite{Hmidi-2024-Lake} with $b\equiv1$, we have
            \begin{equation*}
                \begin{split}
                    \frac{h_1(\theta)}{a_1}\int_{0}^{2\pi}\int_{a_1}^{a_2}e^{i\theta}\cdot\nabla\ln|a_1e^{i\theta}-\rho e^{i\eta}|\rho &=-\frac{h_1(\theta)}{a_1^2}\left(\int_{0}^{\min\{a_1,a_2\}}-\int_{0}^{\min\{a_1,a_1\}}\right)\tau\,d\tau\\
                    &=0,
                \end{split}
            \end{equation*}
            and 
            \begin{equation*}
                \begin{split}
                    \frac{h_2(\theta)}{a_2}\int_0^{2\pi}\int_{a_1}^{a_2}e^{i\theta}\cdot\nabla \ln|a_2e^{i\theta}-\rho e^{i\eta}|\rho&=-\frac{h_2(\theta)}{a_2^2}\left(\int_{0}^{\min\{a_2,a_2\}}-\int_{0}^{\min\{a_1,a_2\}}\right)\tau\,d\tau\\
                    &=-\frac{a_2^2-a_1^2}{2a_2^2}h_2(\theta).
                \end{split}
            \end{equation*}
            Furthermore, by the expression of $h_1$ and $h_2$, we obtain  
            \begin{equation*}
                \begin{split}
                    \int_0^{2\pi}N(a_je^{i\theta}-a_le^{i\eta})h_l(\eta)&=\sum_n\int_0^{2\pi}N(a_j-a_le^{i(\eta-\theta)})h_n^{(l)}\cos(n\eta)\,d\eta\\
                    &=\sum_n\int_0^{2\pi}N(a_j-a_le^{i\eta})h_n^{(l)}\cos(n(\eta+\theta))\\
                    &=\sum_n\left(\int_0^{2\pi}N(a_j-a_l e^{i\eta})\cos(n\eta)\,d\eta\right)h_n^{(l)}\cos(n\theta)\\
                    &=\sum_n\frac{1}{2n}\min\left\{\left(\frac{a_j}{a_l}\right)^n,\left(\frac{a_j}{a_l}\right)^{-n}\right\}h_n^{(l)}\cos(n\theta),
                \end{split}
            \end{equation*}
            for $j,l=1,2$ and here the last equality is due to the Lemma \ref{identity-ln}. 
            Thus we find that
            \begin{equation*}
                \begin{pmatrix}
                    I_1\\
                    I_2
                \end{pmatrix}=\sum_n n M_n(\Om,a_1,a_2)
                \begin{pmatrix}
                    h_n^{(1)}\\
                    h_n^{(2)}
                \end{pmatrix}\sin(n\theta),
            \end{equation*}
            where
            \begin{align*}
                M_{n}(\Om,a_1,a_2)
                =\begin{pmatrix}
                    -\Om+\frac{1}{2n}+\frac{1}{2\pi a_1^2}&-\frac{1}{2n}(\frac{a_1}{a_2})^{n}\vspace{0.25em}\\
                    \frac{1}{2n}(\frac{a_1}{a_2})^{n}&-\Om+\frac{a_2^2-a_1^2}{2a_2^2}-\frac{1}{2n}+\frac{1}{2\pi a_2^2}
                \end{pmatrix}.
            \end{align*}  

            Next we prove that the linearized operator $d_{(r_1,r_2,x_1^1)}(F_1,F_2,G)$ at the equilibrium is a Fredholm type operator with zero index.
            To this end, we make the following decomposition:
            \begin{equation*}
                d_{(r_1,r_2,x_1^1)}(F_1,F_2,G)(\Om,0,0,0)=\mathcal{L}+\mathcal{H}+\mathcal{R},
            \end{equation*}
            where 
            \begin{equation*}
                \mathcal{L}:=
                \begin{pmatrix}
                    (\Om-\frac{1}{2\pi a_1^2})\partial_\theta&0&0\\
                    0&(\Om-\frac{1}{2\pi a_2^2}-\frac{a_2^2-a_1^2}{2a_2^2})\partial_\theta&0\\
                    0&0&\Om
                \end{pmatrix},\quad
                \mathcal{H}:=
                \begin{pmatrix}
                    -\frac{1}{2}\mathbb{H}&0&0\\
                    0&\frac{1}{2}\mathbb{H}&0\\
                    0&0&0
                \end{pmatrix}
            \end{equation*}
            and 
            \begin{equation*}
                \begin{split}
                    \mathcal{R}[h_1,h_2,b]=
                    \begin{pmatrix}
                        0&\frac{a_1\sin\theta}{2a_2}&-\frac{\sin\theta}{2\pi a_1}\vspace{0.25em}\\
                        -\frac{a_1\sin\theta}{2a_2}&0&-\frac{\sin\theta}{2\pi a_2}\vspace{0.25em}\\
                        -\frac{1}{2a_1}&\frac{1}{2a_2}&0
                    \end{pmatrix}
                    \begin{pmatrix}
                        h_1^{(1)}\vspace{0.25em}\\
                        h_1^{(2)}\vspace{0.25em}\\
                        b
                    \end{pmatrix}
                    +
                    \sum_{n=2}^\infty 
                    \begin{pmatrix}
                        \frac{1}{2}\left(\frac{a_1}{a_2}\right)^n h_n^{(2)}\sin(n\theta)\vspace{0.25em}\\
                        -\frac{1}{2}\left(\frac{a_1}{a_2}\right)^nh_n^{(1)}\sin(n\theta)\vspace{0.25em}\\
                        0
                    \end{pmatrix},
                \end{split}
            \end{equation*}
            where $\mathbb{H}$ is the classical $2\pi$-period Hilbert transform defined by
            \begin{equation*}
                \mathbb{H}f(\theta):=\frac{1}{2\pi}\int_0^{2\pi}f(\eta)\cot(\frac{\theta-\eta}{2})\,d\eta
            \end{equation*}
            and we emphasize that for all $n\in \N^*$,
            \begin{equation*}
                \mathbb{H}\cos(n\theta)=\sin(n\theta).
            \end{equation*}

            Moreover, it is readily seen that if $\Om$ satisfies \eqref{Om not equal set}, then the operator $\mathcal{L}:X^\alpha\times X^\alpha\times \R\to Y^\alpha\times Y^\alpha\times \R$ is a Fredholm operator with zero index. 
            
            As for the operator $\mathcal{H}$, it is proved in \cite{Hilbert transform compact} (Proposition 2.1-(iii)) that the Hilbert transform $\mathbb{H}$ as a mapping form $X^\alpha$ into $Y^\alpha$ is a compact operator. Consequently, the operator $\mathcal{H}:X^\alpha\times X^\alpha\times \R\to Y^\alpha\times Y^\alpha\times \R$ is also compact. 
            
            Particularly, the decay of coefficients $(\frac{a_1}{a_2})^n=O(\frac{1}{n^2})$ implies that the $\mathcal{R}$ maps $X^\alpha\times X^\alpha\times \R$ into $Y^{2+\alpha}\times Y^{2+\alpha}\times\R$. Hence, due to the compact embedding $C^{2+\alpha}\hookrightarrow C^\alpha$, we deduce that $\mathcal{R}:X^\alpha\times X^\alpha\times \R\to Y^{\alpha}\times Y^\alpha\times\R$ is a compact operator. 

            In total, under the condition \eqref{Om not equal set}, the linearized operator $d_{(r_1,r_2,x_1^1)}(F_1,F_2,G)(\Om,0,0,0)$ is a compact perturbation of a Fredholm operator with zero index, thus is a Fredholm operator with zero index, see Corollary 5.9 in \cite{Spectral theory}. 
            
            The proof of the lemma is therefore completed. 
        \end{proof}
        The next step is to show the one-dimensional kernel property of the linearized operator $d_{(r_1,r_2,x_1^1)}(F_1,F_2,G)$, which is an important assumption in Crandall-Rabinowitz's Theorem.
        \begin{proposition}\label{proposition of angular velocities is simple}
            For given $0<a_1<a_2$ and $\Om\neq 0$, there exists $N(a_1,a_2)\in\N^*$ such that, for any $n\geq N(a_1,a_2)$, there are two distinct angular velocities
            \begin{equation}\label{angular velocities}
                \Om_{n}^{\pm}=\left(\frac{a_2^2-a_1^2}{4a_2^2}+\frac{1}{4\pi a_2^2}+\frac{1}{4\pi a_1^2}\right)\pm\frac{1}{2}\sqrt{\left(\frac{a_2^2-a_1^2}{2a_2^2}+\frac{1}{2\pi a_2^2}-\frac{1}{2\pi a_1^2}-\frac{1}{n}\right)^2-\frac{1}{n^2}\left(\frac{a_1}{a_2}\right)^{2n}}
            \end{equation}
            for which the matrix $M_n(\Om_n^\pm,a_1,a_2)$ defined in \eqref{def of Mn}, is singular.

            Furthermore, the sequence $\{\Om_n^\pm\}_{n\geq N(a_1,a_2)}$ are strictly monotone satisfying \eqref{Om not equal set} and 
            \begin{equation*}
                \lim_{n\to\infty}\Om_n^\pm\in\left\{\frac{1}{2\pi a_1^2}\ ,\ \frac{1}{2\pi a_2^2}+\frac{a_2^2-a_1^2}{2a_2^2}\right\}
            \end{equation*}
            and for all $n,n'\geq N(a_1,a_2)$,
            \begin{equation*}
                \Om_n^-\neq \Om_{n'}^+\ .
            \end{equation*}
        \end{proposition}
        \begin{proof}
            Due to \eqref{ker d F_1 F_2 G at 0,0,0}, for $[h_1,h_2,b]\in {\rm ker}\left(d_{(r_1,r_2,x_1^1)}(F_1,F_2,G)(\Om,0,0,0)\right)$ and $\Om\neq 0$, then $b$ satisfies
            \begin{equation}\label{ker b}
                    b=\frac{h_1^{(1)}}{2a_1\Om}-\frac{h_1^{(2)}}{2a_2\Om}.
            \end{equation}    
            Plugging \eqref{ker b} into \eqref{ker d F_1 F_2 G at 0,0,0}, it implies that $[h_1,h_2]$ satisfies
            \begin{equation}
                \begin{split}
                    0=&\widetilde{M}_1(\Om,a_1,a_2)
                    \begin{pmatrix}
                        h_1^{(1)}\\
                        h_1^{(2)}
                    \end{pmatrix}\sin\theta+\sum_{n=2}^\infty n M_n(\Om,a_1,a_2)
                    \begin{pmatrix}
                        h_n^{(1)}\\
                        h_n^{(2)}
                    \end{pmatrix}\sin(n\theta),
                \end{split}
            \end{equation}
            where 
            \begin{equation}
                \widetilde{M}_1(\Om,a_1,a_2)=M_1(\Om,a_1,a_2)+\begin{pmatrix}
                    -\frac{1}{4\pi a_1^2\Om}&\frac{1}{4\pi a_1 a_2\Om}\\
                    -\frac{1}{4\pi a_1a_2\Om}&\frac{1}{4\pi a_2^2 \Om}
                \end{pmatrix},
            \end{equation}
            and $M_n(\Om,a_1,a_2)$ is defined in \eqref{def of Mn}. We can compute the determination of $\widetilde{M}_1$ and $M_n,$ for $n\geq 2.$ The first one is 
            \begin{equation*}
                \begin{split}
                    \det(\widetilde{M_1}(\Om,a_1,a_2))&=\Om^2-\left(\frac{a_2^2-a_1^2}{2a_2^2}+\frac{1}{2\pi a_2^2}+\frac{1}{2\pi a_1^2}\right)\Om+\frac{1}{4\pi^2 a_1^2a_2^2}+(\frac{1}{4\pi a_1^2}-\frac{1}{4\pi a_2^2})\\
                    &>\det(M_1(\Om,a_1,a_2)),
                \end{split}
            \end{equation*}
            which is exactly a second order polynomial. Thus we denote the two distinct roots (if exist) of $\det(\widetilde{M}_1(\Om,a_1,a_2))$ by $\Om_1^\pm$.
            
            Furthermore, we can directly compute the determination of $M_{n}$:
            \begin{align*}
                \det( M_{n}(\Om,a_1,a_2))&=\Om^2-\left(\frac{a_2^2-a_1^2}{2a_2^2}+\frac{1}{2\pi a_2^2}+\frac{1}{2\pi a_1^2}\right)\Om+\frac{1}{2n}\left(\frac{a_2^2-a_1^2}{2a_2^2}+\frac{1}{2\pi a_2^2}-\frac{1}{2\pi a_1^2}\right)\\
                &\quad +\frac{1}{4n^2}\left(\left(\frac{a_1}{a_2}\right)^{2n}-1\right)+\frac{1}{2\pi a_1^2}\left(\frac{a_2^2-a_1^2}{2a_2^2}+\frac{1}{2\pi a_2^2}\right),
            \end{align*}
            which is a second order polynomial as well. Then we also denote the two distinct roots (if exist) of $\det(M_n(\Om,a_1,a_2))$ by $\Om_n^\pm$.
            The discriminant of the preceding  second order polynomial is given by
            \begin{align}
                \Delta_{n}&=\left(\frac{a_2^2-a_1^2}{2a_2^2}+\frac{1}{2\pi a_2^2}-\frac{1}{2\pi a_1^2}-\frac{1}{n}\right)^2-\frac{1}{n^2}\left(\frac{a_1}{a_2}\right)^{2n}.
            \end{align} 
            
            Next, we focus on the existence of the solutions of $\det(M_n(\Om,a_1,a_2))=0$. 
            Notice that 
            \begin{equation*}
                \frac{a_2^2-a_1^2}{2a_2^2}+\frac{1}{2\pi a_2^2}-\frac{1}{2\pi a_1^2}=\frac{1}{2}\left(1-\frac{1}{\pi a_1^2}\right)\left(1-\frac{a_1^2}{a_2^2}\right).
            \end{equation*}
            Therefore, it is seen that if $\pi a_1^2=1$, then 
            \begin{equation}
                \Delta_n =\frac{1}{n^2}-\frac{1}{n^2}\left(\frac{a_1}{a_2}\right)^{2n}>0,\quad {\rm for\ all}\ n\geq 2,
            \end{equation}
            and if $\pi a_1^2\neq 1$, then 
            \begin{equation}
                \Delta_n  \to \left(\frac{a_2^2-a_1^2}{2a_2^2}+\frac{1}{2\pi a_2^2}-\frac{1}{2\pi a_1^2}\right)^2>0,\quad {\rm as}\ n\to\infty. 
            \end{equation}
            Thus, no matter $\pi a_1^2$ is equal to $1$ or not, there always exists $N_1(a_1,a_2)\in \mathbb{N}^*$ such that 
            \begin{align}
                \Delta_{n}>0, \quad {\rm for \ all}\ n\geq N_1(a_1,a_2).
            \end{align} 
            Therefore, the polynomial $\det (M_{n})$ admits two distinct roots $\Om_{n}^{\pm}$ being as given in the statement of the proposition and we write that 
            \begin{itemize}
                \item For $\pi a_1^2>1$, then $\frac{a_2^2-a_1^2}{2a_2^2}+\frac{1}{2\pi a_2^2}-\frac{1}{2n}>\frac{1}{2\pi a_1^2}+\frac{1}{2n}$, for $n\geq N_2(a_1,a_2)$, and 
                \begin{equation}
                    \begin{split}
                        \Om_{n}^+&=\left(\frac{a_2^2-a_1^2}{4a_2^2}+\frac{1}{4\pi a_2^2}-\frac{1}{2n}\right)+r_{n},\\
                        \Om_{n}^-&=\left(\frac{1}{4\pi a_1^2}+\frac{1}{2n}\right)-r_{n}.
                    \end{split}
                \end{equation}
                \item For $\pi a_1^2\leq 1$, then $\frac{a_2^2-a_1^2}{2a_2^2}+\frac{1}{2\pi a_2^2}<\frac{1}{2\pi a_1^2}$  and for all $n\geq N_1(a_1,a_2)$, we have
                    \begin{equation}
                        \begin{split}
                            \Om_{n}^+&=\left(\frac{1}{4\pi a_1^2}+\frac{1}{2n}\right)+r_{n},\\
                            \Om_{n}^-&=\left(\frac{a_2^2-a_1^2}{4a_2^2}+\frac{1}{4\pi a_2^2}-\frac{1}{2n}\right)-r_{n},
                        \end{split}
                    \end{equation}
            \end{itemize}
            where we set 
            \begin{equation}
                \begin{split}
                    r_{n}&=\frac{1}{2}\left|\frac{a_2^2-a_1^2}{2a_2^2}+\frac{1}{2\pi a_2^2}-\frac{1}{2\pi a_1^2}-\frac{1}{n}\right|\\
                    &\quad \times \left(\sqrt{1-\frac{\frac{1}{n^2}(a_1/a_2)^{2n}}{\left(\frac{a_2^2-a_1^2}{2a_2^2}+\frac{1}{2\pi a_2^2}-\frac{1}{2\pi a_1^2}-\frac{1}{n}\right)^2}}-1\right)  .         
                \end{split}
            \end{equation}
            Moreover, it is not difficult to see that
            \begin{align}
                r_{n}=O(\frac{1}{n^3}),
            \end{align}
            consequently, 
            \begin{itemize}
                \item For $\pi a_1^2>1$, then 
                    \begin{equation}
                        \begin{split}
                            \Om_{n+1}^+-\Om_{n}^+&=\frac{1}{2(n+1)n}+O(\frac{1}{n^3}),\\
                            \Om_{n+1}^--\Om_{n}^-&=-\frac{1}{2(n+1)n}+O(\frac{1}{n^3}).
                        \end{split}
                    \end{equation}
                \item For $\pi a_1^2\leq 1$, then 
                    \begin{equation}
                        \begin{split}
                            \Om_{n+1}^+-\Om_{n}^+&=-\frac{1}{2(n+1)n}+O(\frac{1}{n^3}),\\
                            \Om_{n+1}^--\Om_{n}^-&=\frac{1}{2(n+1)n}+O(\frac{1}{n^3}).
                        \end{split}
                    \end{equation}
            \end{itemize}

            All in all, we deduce that the sequence $\{\Om_{n}^{\pm}\}$ is strictly monotone for $n\geq N_3(a_1,a_2)$. 
            Obviously, we have $\Om_n^-< \frac{a_2^2-a_1^2}{4a_2^2}+\frac{1}{4\pi a_2^2}+\frac{1}{4\pi a_1^2}< \Om_n^+$, for all $n\geq N_3(a_1,a_2)$. 
            Furthermore, there holds
            \begin{equation}\label{limit of Om plus}
                \begin{split}
                    \lim_{n\to \infty}\Om_{n}^+&=\left(\frac{a_2^2-a_1^2}{4a_2^2}+\frac{1}{4\pi a_2^2}+\frac{1}{4\pi a_1^2}\right)+\frac{1}{2}\left|\frac{a_2^2-a_1^2}{2a_2^2}+\frac{1}{2\pi a_2^2}-\frac{1}{2\pi a_1^2}\right|\vspace{0.25em}\\    
                    &=
                    \begin{cases}
                        \frac{a_2^2-a_1^2}{2a_2^2}+\frac{1}{2\pi a_2^2}~,&{\rm if}\ \pi a_1^2>1,\vspace{0.25em}\\
                        \frac{1}{2\pi a_1^2}~,& {\rm if}\ \pi a_1^2\leq 1,
                    \end{cases}
                \end{split}
            \end{equation}
            and
            \begin{equation}\label{limit of Om minus}
                \begin{split}
                    \lim_{n\to\infty}\Om_{n}^-&=\left(\frac{a_2^2-a_1^2}{4a_2^2}+\frac{1}{4\pi a_2^2}+\frac{1}{4\pi a_1^2}\right)-\frac{1}{2}\left|\frac{a_2^2-a_1^2}{2a_2^2}+\frac{1}{2\pi a_2^2}-\frac{1}{2\pi a_1^2}\right|\vspace{0.25em}\\   
                    &=
                    \begin{cases}
                        \frac{1}{2\pi a_1^2}~,& {\rm if}\ \pi a_1^2>1,\vspace{0.25em}\\
                        \frac{a_2^2-a_1^2}{2a_2^2}+\frac{1}{2\pi a_2^2}~,&{\rm if}\ \pi a_1^2\leq 1.
                    \end{cases}
                \end{split}
            \end{equation}
            It implies that $\{\Om_n^\pm\}_{n\geq N_3(a_1,a_2)}$ satisfies the assumption \eqref{Om not equal set} in Lemma \ref{differential at 0,0,0}.

            Finally, we claim that there is no spectral collisions for $n$ large enough.

            In fact, $\{\Om_n^{\pm}\}_{n\geq N_3(a_1,a_2)}$ are all distinct from each other. As for $\{\Om_n^\pm\}_{n<N_3(a_1,a_2)}$, we set
            \begin{equation*}
                \delta:=\min\left\{\left|\Om_j^\pm-\frac{a_2^2-a_1^2}{2a_2^2}-\frac{1}{2\pi a_2^2}\right|,\ \left|\Om_j^\pm-\frac{1}{2\pi a_1^2}\right|:\ \Om_j^\pm\ {\rm satisfies}\ \eqref{Om not equal set},\ j<N_3(a_1,a_2)\right\}. 
            \end{equation*}
            Combining the strictly monotonic properties of $\{\Om_n^\pm\}_{n\geq N_3(a_1,a_2)}$ and \eqref{limit of Om plus}, \eqref{limit of Om minus}, there exists $N_4(a_1,a_2)\geq N_3(a_1,a_2)$, such that 
            \begin{equation*}
                \max\left\{\left|\Om_n^\pm-\frac{1}{2\pi a_1^2}\right|,\ \left|\Om_n^\pm-\frac{a_2^2-a_1^2}{2a_2^2}-\frac{1}{2\pi a_2^2}\right| ~:~n\geq N_4(a_1,a_2) \right\}\leq \frac{\delta}{2}.
            \end{equation*} 
            Thus for any $n\geq N_4(a_1,a_2)$, there holds $\Om_n^\pm\neq \Om_j^\pm$, for all $j\neq n$. The proof is therefore completed.
            
        \end{proof}

    \subsection{Proof of Theorem \ref{main theorem 2}}
        Finally, we check the last prerequisites that we need to apply Crandall-Rabinowitz's Theorem to, eventually, conclude the proof for Theorem \ref{main theorem 2}.
        \begin{proposition}
            Let $\alpha\in(0,1),\ 0<a_1<a_2$ and $N(a_1,a_2)\in \N^*$ be as in Proposition \ref{proposition of angular velocities is simple}. Then, for all $n\geq N(a_1,a_2)$ and $\Om_n^\pm$ defined in \eqref{angular velocities}, the kernel of the linearized operator 
            $$d_{(r_1,r_2,x_1^1)}(F_1,F_2,G)(\Om_{n}^\pm,0,0,0):X^\alpha\times X^\alpha\times\R\to Y^\alpha\times Y^\alpha\times\R$$
            is one-dimension and generated by $(r_0^1,r_0^2,0)$, where 
            \begin{equation}
                \begin{pmatrix}
                    r_0^1\\
                    r_0^2
                \end{pmatrix}
                :\theta\mapsto \begin{pmatrix}
                    \frac{1}{2n}(\frac{a_1}{a_2})^{n}\vspace{0.5em}\\
                    -\Om_{n}^{\pm}+\frac{1}{2n}+\frac{1}{2\pi a_1^2}
                \end{pmatrix}\cos(n\theta).
            \end{equation}
            
            Moreover, the range of linearized operator $d_{(r_1,r_2,x_1^1)}(F_1,F_2,G)(\Om_n^\pm,0,0,0)$ is closed and of co-dimension one.
            
            Furthermore, the transversality condition is satisfied, that is
            $$\partial_\Om d_{(r_1,r_2,x_1^1)}(F_1,F_2,G)(\Om_{n}^{\pm},0,0,0)[r_0^1,r_0^2,0]\notin \Ran \left(d_{(r_1,r_2,x_1^1)}(F_1,F_2,G)(\Om_{n}^\pm,0,0,0)\right).$$
        \end{proposition}
        \begin{proof}
            Due to the analysis given in Proposition \ref{proposition of angular velocities is simple}, we deduce for $n\geq N(a_1,a_2)$, $\Om_n^\pm$ satisfies \eqref{Om not equal set}, and 
            \begin{equation*}
                \det(M_n(\Om_n^\pm,a_1,a_2))=0
            \end{equation*}
            and for all $j\neq n$, 
            \begin{equation*}
                \det(M_j(\Om_n^\pm,a_1,a_2))\neq 0.
            \end{equation*}
            Therefore, in view of the representation \eqref{differential at 0,0,0} of the linearized operator $d_{(r_1,r_2,x_1^1)}(F_1,F_2,G)$ at $(\Om_n^\pm,0,0,0)$, one deduces that the kernel of this operator is one-dimensional generated by the vector $(r_0^1,r_0^2,0)$ as defined in the statement of the proposition.

            Due to Proposition \ref{proposition of angular velocities is simple}, we find that $\Om_n^\pm$ satisfies \eqref{Om not equal set}. Therefore, by the virtue of Lemma \ref{differential at 0,0,0}, the Fredholm property for the linearized operator $d_{(r_1,r_2,x_1^1)}(F_1,F_2,G)$ at $(\Om_n^\pm,0,0,0)$ follows. 
            Hence we deduce that the range of this linearized operator is closed and of co-dimension one as a result of the one-dimensional property of its kernel.

            Next, let us give a more precise form about the range of the linearized operator. Therefore, we equip the Hilbert space $Y^\alpha\times Y^\alpha\times \R$ with the following $l^2$-scalar product: for two vectors with the following form
            \begin{equation*}
                \begin{split}
                    (H_1,H_2,A)=\left(\sum_{n=1}a_n\sin(n\theta),\sum_{n=1}b_n \sin(n\theta),A\right)\in Y^\alpha\times Y^\alpha\times \R,\vspace{0.25em}\\
                    (K_1,K_2,B)=\left(\sum_{n=1}c_n\sin(n\theta),\sum_{n=1}d_n \sin(n\theta),B\right)\in Y^\alpha\times Y^\alpha\times\R,
                \end{split}
            \end{equation*}
            there holds
            \begin{equation}
                \left\langle (H_1,H_2,A),(K_1,K_2,B)\right\rangle:=\sum_{n=1}a_n c_n+b_nd_n+A\cdot B.
            \end{equation}
            Furthermore, define 
            \begin{equation*}
                y_0=\left(-\frac{1}{2n}\left(\frac{a_1}{a_2}\right)^{n}\sin(n\theta),\left(-\Om_{n}^\pm+\frac{1}{2n}+\frac{1}{2\pi a_1^2}\right)\sin(n\theta),0\right).
            \end{equation*}
            It is obvious that 
            \[ 
                (M_n(\Om_n^\pm,a_1,a_2))^T
                \begin{pmatrix}
                    -\frac{1}{2n}\big(\frac{a_1}{a_2}\big)^{n}\vspace{0.5em}\\
                    -\Om_{n}^\pm+\frac{1}{2n}+\frac{1}{2\pi a_1^2}
                \end{pmatrix}=0,
            \]
            where $M^T$ denotes the transpose of the matrix $M$. Then for any $(h_1,h_2,b)\in X^\alpha\times X^\alpha\times\R$ given by
            \begin{equation*}
                h_1(\theta)=\sum_{n=1}h_n^{(1)}\cos(n\theta),\quad h_2(\theta)=\sum_{n=1}h_n^{(2)}\cos(n\theta),
            \end{equation*}
            it holds that
            \begin{equation*}
                \begin{split}
                    &\left\langle d_{(r_1,r_2,x_1^1)}(F_1,F_2,G)(\Om_{n}^\pm,0,0,0)[h_1,h_2,b],y_0 \right\rangle\\
                    &=-n\left\langle \begin{pmatrix}
                        M_n(\Om_n^\pm,a_1,a_2)
                        \begin{pmatrix}
                            h_n^{(1)}\vspace{0.25em}\\ 
                            h_n^{(2)}
                        \end{pmatrix}\sin(n\theta)\\
                        b
                    \end{pmatrix}
                    , 
                    \begin{pmatrix}
                        (\Om_{n}^\pm-\frac{1}{2\pi a_1^2})\sin(n\theta)\vspace{0.25em}\\
                        \frac{1}{2n}\big(\frac{a_1}{a_2}\big)^{n}\sin(n\theta)\\
                        0
                    \end{pmatrix} \right\rangle \\
                    &=-n\left\langle
                        \begin{pmatrix}
                            h_n^{(1)}\sin(n\theta)\vspace{0.25em}\\ 
                            h_n^{(2)}\sin(n\theta)\\
                            b
                        \end{pmatrix},
                        \begin{pmatrix}
                            (M_n(\Om_n^\pm,a_1,a_2))^T
                        \begin{pmatrix}
                            (\Om_{n}^\pm-\frac{1}{2\pi a_1^2})\vspace{0.25em}\\
                            \frac{1}{2n}\big(\frac{a_1}{a_2}\big)^{n}
                        \end{pmatrix}\sin(n\theta) \\
                        0
                        \end{pmatrix}
                        \right\rangle\\
                    &=0,
                \end{split}
            \end{equation*}
            which implies that $\Ran\left( d_{r_1,r_2,x_1^1}(F_1,F_2,G)(\Om_n^\pm,0,0,0)\right)\subset \Span_\alpha^\perp (y_0)$, 
            where 
            \begin{equation*}
                \Span_\alpha^\perp (y_0):=\left\{y\in Y^\alpha\times Y^\alpha\times\R :\langle y,y_0\rangle=0\right\}.
            \end{equation*}
            Since $\Span_\alpha^\perp(y_0)$ is of co-dimension one in $Y^\alpha\times Y^\alpha\times \R$, it follows from the Fredholm property of $d_{(r_1,r_2,x_1^1)}(F_1,F_2,G)(\Om_n^\pm,0,0,0)$ that 
            \begin{equation}
                \Ran \left(d_{(r_1,r_2,x_1^1)}(F_1,F_2,G)(\Om_{n}^\pm,0,0,0)\right)={\rm span}_\alpha^\perp(y_0).
            \end{equation}
            It remains to prove the transversality condition. To this end, noting that  
            \begin{equation*}
                \begin{split}
                    &\partial_\Om d_{(r_1,r_2,x_1^1)}(F_1,F_2,G)(\Om_{n}^{\pm},0,0,0)[r_0^1,r_0^2,0]\\
                    &=n\left(-\frac{1}{2n}\left(\frac{a_1}{a_2}\right)^n \sin(n\theta),\left(\Om_n^\pm-\frac{1}{2n}-\frac{1}{2\pi a_1^2}\right)\sin(n\theta),0\right),
                \end{split}
            \end{equation*}
            we obtain that
            \begin{equation*}
                \begin{split}
                    &\langle \partial_\Om d_{(r_1,r_2,x_1^1)}(F_1,F_2,G)(\Om_{n}^{\pm},0,0,0)[r_0^1,r_0^2,0],y_0\rangle\\
                    &=-n\left(\left(\Om_n^\pm-\frac{1}{2n}-\frac{1}{2\pi a_1^2}\right)^2-\frac{1}{4n^2}\left(\frac{a_1}{a_2}\right)^{2n}\right)\\
                    &=-n\left(\Om_n^\pm-\frac{1}{2n}-\frac{1}{2\pi a_1^2}\right)\left(2\Om_n^\pm-\frac{a_2^2-a_1^2}{2a_2^2}-\frac{1}{2\pi a_1^2}-\frac{1}{2\pi a_2^2}\right).
                \end{split}
            \end{equation*}
            By the virtue of expression of $\Om_n^\pm$ given in \eqref{angular velocities} and $\Delta_n>0$, as long as $n\geq N(a_1,a_2)$, it shows that $2\Om_n^\pm-\frac{a_2^2-a_1^2}{2a_2^2}-\frac{1}{2\pi a_1^2}-\frac{1}{2\pi a_2^2}\neq 0$.
            As for the term $\Om_n^\pm-\frac{1}{2n}-\frac{1}{2\pi a_1^2}$, direct computation shows that
            \begin{equation*}
                \begin{split}
                    \Om_n^\pm-\frac{1}{2n}-\frac{1}{2\pi a_1^2}=&\frac{1}{2}\left(\frac{a_2^2-a_1^2}{2a_2^2}+\frac{1}{2\pi a_2^2}-\frac{1}{2\pi a_1^2}-\frac{1}{n}\right)\\
                    &\pm \frac{1}{2}\sqrt{\left(\frac{a_2^2-a_1^2}{2a_2^2}+\frac{1}{2\pi a_2^2}-\frac{1}{2\pi a_1^2}-\frac{1}{n}\right)^2-\frac{1}{n^2}\left(\frac{a_1}{a_2}\right)^{2n}}\\
                    \neq&0.
                \end{split}
            \end{equation*}
            As having expected, we have 
            \begin{equation*}
                \begin{split}
                    \partial_\Om d_{(r_1,r_2,x_1^1)}(F_1,F_2,G)(\Om_{n}^{\pm},0,0,0)[r_0^1,r_0^2,0]&\notin \Span_\alpha^\perp (y_0)\\
                    &=\Ran(\partial_\Om d_{(r_1,r_2,x_1^1)}(F_1,F_2,G)(\Om_{n}^{\pm},0,0,0)).
                \end{split}
            \end{equation*}
            It is now obvious that the proposition holds.
        \end{proof}
        
\subsection*{Acknowledgments:}

    This work was supported by  National Key R\&D Program of China (Grant 2022YFA1005602) and NNSF of China (grant No.12371212).

\phantom{s}
\thispagestyle{empty}

\end{document}